\documentclass{amsart}

\usepackage{amsfonts}
\usepackage{amssymb}
\usepackage{latexsym}
\usepackage{multirow}

\newtheorem{theorem}{Theorem}[section]
\newtheorem{lemma}[theorem]{Lemma}
\newtheorem{proposition}[theorem]{Proposition}

\theoremstyle{definition}
\newtheorem{definition}[theorem]{Definition}

\newtheorem{remark}[theorem]{Remark}

\newtheorem{example}[theorem]{Example}

\numberwithin{equation}{section}

\newcommand{\R}{\mathbb{R}}
\newcommand{\Z}{\mathbb{Z}}
\newcommand{\F}{\mathbb{F}}
\newcommand{\La}{\Lambda}
\newcommand{\Sp}{\mathbb{S}}

\newcommand{\p}{^{\prime}}
\newcommand{\co}{\colon\thinspace}

\begin{document}

\title{$\Gamma$-extensions of the spectrum of an orbifold}

\author{Carla Farsi}
\address{Department of Mathematics, University of Colorado at Boulder,
    Campus Box 395, Boulder, CO 80309-0395}
\email{farsi@euclid.colorado.edu}

\author{Emily Proctor}
\address{Department of Mathematics, Middlebury College, Middlebury, VT 05753}
\email{eproctor@middlebury.edu}

\author{Christopher Seaton}
\address{Department of Mathematics and Computer Science,
    Rhodes College, 2000 N. Parkway, Memphis, TN 38112}
\email{seatonc@rhodes.edu}

\begin{abstract}

We introduce the $\Gamma$-extension of the spectrum of the Laplacian
of a Riemannian orbifold, where $\Gamma$ is a finitely generated discrete group.
This extension, called the \emph{$\Gamma$-spectrum,} is the union of the Laplace spectra
of the $\Gamma$-sectors of the orbifold, and hence constitutes a Riemannian invariant that
is directly related to the singular set of the orbifold.
We compare the $\Gamma$-spectra of known examples of isospectral pairs and
families of orbifolds and demonstrate that in many cases, isospectral orbifolds need not be
$\Gamma$-isospectral.  We additionally prove a version of Sunada's theorem that allows us to
construct pairs of orbifolds that are $\Gamma$-isospectral for any choice of $\Gamma$.

\end{abstract}



\maketitle

\tableofcontents

% XXXXXXXXXXXXXXXXXXXXXXXXXXXXXXXXXXXXXXXXXXXXXXXXXXXXXXXXXXXXXX
% XXXXXXXXXX    SECTION: INTRODUCTION
% XXXXXXXXXXXXXXXXXXXXXXXXXXXXXXXXXXXXXXXXXXXXXXXXXXXXXXXXXXXXXX

\section{Introduction}
\label{sec-Intro}

A central question of spectral geometry concerns the extent to which the Laplace spectrum of an orbifold influences
its geometry and topology, and vice versa.   For example, if two orbifolds have the same spectrum, they
must have the same volume and dimension \cite{farsi}.  On the other hand, there are many examples of
isospectal, nonisometric orbifolds.  Most of the early examples of isospectral orbifolds were manifolds
\cite{milnor, vigneras, ikedalensspace, sunada, gordon, schueth}.  More recently, however, attention
has turned to the study of the spectrum of orbifolds having nontrivial singular sets and the interplay
between the spectrum and the singular set.   In 2006, Shams, Stanhope, and Webb
produced arbitrarily large finite families of isospectral orbifolds \cite{ssw}.  Any pair of orbifolds in
a given family contains points with nonisomorphic isotropy groups, but all orbifolds in the family have
maximal isotropy groups of the same order.   Rossetti, Schueth, and Weilandt
have since produced examples of pairs isospectral orbifolds having different maximal isotropy order
\cite{rsw}.  Working independently, both Sutton, and the second author working with
Stanhope, found examples of continuous families of isospectral, nonisometric orbifolds
\cite{sutton, procstan}. More recently, Shams produced examples of pairs of isospectral, nonisometric
orbifold lens spaces \cite{shams}.  In the positive direction, Dryden, Gordon,
Greenwald, and Webb constructed the asymptotic expansion for the heat trace of
a general compact Riemannian orbifold in such a way that the contribution of each piece of the
singular set to the heat invariants is evident \cite{dggw}.  They used their results to show
that the spectrum can distinguish orbifolds within certain classes of two-dimensional orbifolds.
Dryden and Strohmaier showed that for a compact orientable hyperbolic orbisurface,
the numbers and types of singular points as well as the length spectrum of the orbifold
are completely determined by the Laplace spectrum \cite{drystroh}.
This was shown independently by Doyle and Rossetti, who in \cite{DoyleRossetti} also
proved an extension to the case of compact hyperbolic orbisurfaces that are not necessarily connected
or orientable.
In \cite{proctor}, the second author proved that any isospectral collection of orbifolds with
sectional curvature uniformly bounded below and having only isolated singular points contains
only finitely many orbifold category homeomorphism types.

In this paper, we address the question of the relationship between
the spectrum of an orbifold $O$ and its singular set by introducing
the \textit{$\Gamma$-spectrum} of $O$.  The $\Gamma$-spectrum is the
$\Gamma$-extension of the spectrum of the Laplacian in the sense of
\cite{tamanoi2}, meaning that it is an application of the spectrum
to the \emph{$\Gamma$-sectors} of $O$.  Originally introduced in
\cite{tamanoi1} for global quotients and \cite{farseanonvanish} for
general orbifolds, the $\Gamma$-sectors of $O$ consist of a disjoint
union of orbifolds of varying dimensions including a copy of $O$ as
well as other components that finitely cover the singular set of
$O$.  Special cases include the \emph{inertia orbifold} when $\Gamma
= \Z$ and the \emph{multisectors} when $\Gamma$ is a free group; see
e.g. \cite{ademleidaruan}.  The orbifold of $\Gamma$-sectors can be
thought of an ``unraveling" of the singular set of $O$ into distinct
orbifolds, where the group $\Gamma$ determines the type of
singularities that are unraveled. In this sense, the
$\Gamma$-spectrum includes the ordinary spectrum of $O$ as well as,
approximately speaking, the spectra of various components of the
singular locus of $O$. The technique of extending an orbifold
invariant by considering its $\Gamma$-extension has been studied in
the case of Euler characteristics and related orbifold invariants;
here, we apply this technique to the case of an invariant of a
Riemannian orbifold, the spectrum of the Laplacian.

The definition of an orbifold varies considerably from author
to author and discipline to discipline, based on the features of the
orbifold structure that are under consideration.  In particular, in
Riemannian geometry, orbifolds are usually assumed to be effective,
and hence can be presented as quotient orbifolds; see Section~\ref{sec-Background}.
Considering the $\Gamma$-spectrum of a noneffective orbifold leads
to trivial examples that are contrary to the spirit of the
investigation presented here; see Section~\ref{subsec-BackSpectrum}
and in particular Examples~\ref{ex-MtrivZ2} and \ref{ex-MtrivZ3D6}.  Hence, we consider the
$\Gamma$-spectrum to be most interesting when applied to an
effective Riemannian orbifold.  For this reason, though we do explain the
(direct) generalizations of the definitions presented here to general
orbifolds presented by groupoids, we otherwise restrict our attention to
quotient orbifolds.
Note that the $\Gamma$-sectors of a nontrivial orbifold $O$ may include
noneffective orbifolds even when $O$ is assumed effective, and thus our discussion requires the
consideration of noneffective orbifolds.  If $O$ is a quotient orbifold, however,
the $\Gamma$-sectors of $O$ are quotient orbifolds as well, even when they are
not effective.

This paper is organized as follows.  In Section~\ref{sec-Background}, we recall the relevant
definitions and fix notation.  The $\Gamma$-sectors are defined in Section \ref{subsec-BackSectors},
and the $\Gamma$-spectrum is defined in Section \ref{subsec-BackSpectrum}.  In Section \ref{subsec-HeatKer},
we discuss some immediate consequences of the definition of the $\Gamma$-spectrum.
Section~\ref{sec-RSW} contains a description of the $\Gamma$-sectors of several collections
of known isospectral, nonisometric orbifolds in order to determine whether they are also
$\Gamma$-isospectral.  In Section~\ref{sec-Sunada}, we prove a Sunada-type theorem for
$\Gamma$-isospectrality and exhibit examples of nonisometric orbifolds that are
nontrivially $\Gamma$-isospectral for different choices of $\Gamma$.

\section*{Acknowledgements}

The authors would like to thank the referee for detailed corrections and helpful suggestions.
The first author would like to thank the University of Florence for hospitality during the completion of
this manuscript.  The third author was partially funded by a Rhodes College Faculty Development Endowment
Grant.

\bigskip
% XXXXXXXXXXXXXXXXXXXXXXXXXXXXXXXXXXXXXXXXXXXXXXXXXXXXXXXXXXXXXX
% XXXXXXXXXX    SECTION: BACKGROUND
% XXXXXXXXXXXXXXXXXXXXXXXXXXXXXXXXXXXXXXXXXXXXXXXXXXXXXXXXXXXXXX

\section{Background and definitions}
\label{sec-Background}

Let $O$ be an $n$-dimensional orbifold.  We will be primarily interested in the case
that $O$ can be presented as a \emph{quotient orbifold}, i.e. $O$ is given by $G\backslash M$ with orbifold structure
given by the translation groupoid $G\ltimes M$ where
$M$ is a smooth manifold and $G$ is a compact Lie group acting
with finite isotropy groups (on the left) on $M$.  If $G$ is finite, then we say $O$ is
a \emph{global quotient orbifold}.  More generally, an orbifold is defined by a Morita
equivalence class of a proper \'{e}tale Lie groupoid; see \cite[Definition 1.48]{ademleidaruan} and the following paragraph.  A proper \'{etale} Lie groupoid can be thought of as an atlas for the orbifold it represents, and the orbifold is given by the orbit space of the groupoid under the action of its arrows.
In either case, $O$ consists of a
second countable Hausdorff space $\mathbb{X}_O$, the \emph{underlying space of $O$},
that is covered by \emph{orbifold charts} of the form $\{ V_x, G_x, \pi_x \}$
where $V_x$ is diffeomorphic to $\R^n$, $G_x$ is a finite group acting linearly on
$V_x$, and $\pi_x\co V_x \to \mathbb{X}_O$ induces a homeomorphism of $G_x\backslash V_x$ onto an open subset
of $\mathbb{X}_O$.  If $O$ is \emph{effective},
i.e. for each orbifold chart $\{ V_x, G_x, \pi_x\}$ the group $G_x$ acts effectively on $V_x$,
then it is well known that $O$ can be
presented as a quotient orbifold using the \emph{frame bundle construction}, and that $G$ can be taken to be $O(n)$;
see \cite[Theorem 1.23]{ademleidaruan}.  It has been conjectured that all orbifolds can be presented as quotient
orbifolds; see \cite[p. 27]{ademleidaruan}.

Equivalence of orbifolds is subtle and most easily described in terms of
Morita equivalence of groupoid presentations.  Two groupoids $\mathcal{G}$ and $\mathcal{G}^\prime$ are \textit{Morita equivalent} if there is a groupoid $\mathcal{H}$ and a chain of groupoid equivalences $\mathcal{G}\leftarrow\mathcal{H}\rightarrow\mathcal{G}^\prime$.  (See \cite[Definition 1.42]{ademleidaruan} for the definition of groupoid equivalence.)  Each groupoid equivalence induces a homeomorphism between underlying spaces of the orbifolds that the groupoids represent, preserving the isomorphism class of the isotropy group of each point in the orbifold.  By identifying the underlying spaces via these homeomorphisms, one may think of $\mathcal{H}$ as an orbifold atlas that refines the orbifold atlases corresponding to $\mathcal{G}$ and $\mathcal{G}^{\prime}$.   Note in particular that for two groupoids $\mathcal{G}$ and $\mathcal{G}^\prime$ to be Morita equivalent, there need not be a map directly from $\mathcal{G}$ to $\mathcal{G}^{\prime}$.  Intuitively, this corresponds to the fact that the atlas corresponding to $\mathcal{G}$ might not be fine enough to locally define a map.

If $O$ and $O^{\prime}$ are quotient orbifolds represented by $G\ltimes M$ and $H\ltimes N$ respectively, then an \textit{equivariant map} from $O$ to $O^\prime$ consists of a homomorphism $\varphi:G\to H$ and a smooth map $f:M\to N$ such that $f(gx) = \varphi(g)f(x)$ for all $g\in G$ and $x\in M$.   It has been shown that if we restrict our attention to the category of smooth translation groupoids $G\ltimes M$ and equivariant maps, then Morita equivalence is still a well-defined equivalence relation on this category \cite{pronkscull}.  Thus when we restrict our attention to orbifolds that can be represented as quotients, we need only concern ourselves with equivariant groupoid equivalences, and with Morita equivalences via smooth translation groupoids.  Hence we define a \textit{diffeomorphism} between quotient orbifolds to be an equivariant groupoid equivalence.  We note that by Proposition 3.5 in \cite{pronkscull}, every diffeomorphism between quotient orbifolds is given as a composition of quotient and inductive groupoid equivalences.  We say that two quotient orbifolds $O$ and $O^{\prime}$ represented by $G\ltimes M$ and $H\ltimes N$ respectively are \textit{diffeomorphic} if they can be connected by a Morita equivalence
\begin{equation*}
G\ltimes M\xleftarrow{(\varphi,f)}K\ltimes Z\xrightarrow{(\omega,h)}H\ltimes N
\end{equation*}
where $(\varphi,f)$ and $(\omega,h)$ are equivariant groupoid equivalences.

Classically, a Riemannian metric on an orbifold $O$ has been defined via charts.  For each chart $\{V_x,G_x,\pi_x\}$, let $g_x$ be a $G_x$-invariant Riemannian metric on $V_x$.   Patching the charts together via a partition of unity gives a \emph{Riemannian structure} on $O$.
If $O$ is presented as a quotient $G \ltimes M$
for $G$ a compact Lie group, then any $G$-invariant metric on $M$
induces a metric on the orbifold $O$, and any metric on $O$ is
induced by such a metric; see \cite[Proposition 2.1]{stanuribe}. In
particular, given a point $x \in M$, a slice $W_x$ at $x$ for the
$G$-action on $M$ induces a local chart $\{ W_x, G_x, \pi_x \}$ for
the orbifold $O$ at $Gx$; see \cite[Definition 2.3.1]{duistermaatKolk}.
The $G$-invariant metric on $M$
restricts to a $G_x$-invariant metric on $W_x$. Note that distinct
$G$-invariant metrics on $M$ may correspond to the same metric on
$O$.

When studying the Riemannian geometry of orbifolds, it is common and
natural to restrict to the consideration of effective orbifolds.
Indeed, the Riemannian structure of a non-effective orbifold $O$ is
identical to that of its \emph{effectivisation} $O_{\textit{eff}}$; see \cite[Definition 2.33]{ademleidaruan}, except for a minor change to integration on the orbifold, see Section~\ref{subsec-HeatKer}.
If $O$ is a noneffective orbifold with quotient presentation $G\ltimes M$
and $K$ is the (necessarily normal) subgroup of $G$ that acts trivially
on $M$, then $O_{\textit{eff}}$ is the effective orbifold presented by
$G/K \ltimes M$, and it is clear that the Riemiannian structures of
$O$ and $O_{\textit{eff}}$ coincide.  However, when considering
the $\Gamma$-sectors and $\Gamma$-spectrum, we will see that a
noneffective group action may arise in the $\Gamma$-sectors, even
in the case that $O$ is effective. For this reason, we make the
following.

\begin{definition}
Let $O=(G\ltimes M, g)$ and $O^\prime=(H\ltimes N,g^{\prime})$ be Riemannian quotient orbifolds.
Let $\bar g$ and $\bar g^{\prime}$ be correponding invariant metrics on $M$ and $N$ respectively.
An \emph{isometry} from $O$ to $O^\prime$ is an equivariant groupoid equivalence $(\varphi,f): G\ltimes M \to H\ltimes N$ such
that $f^\ast \bar g^\prime = \bar g$.  We say that $O$ and $O^\prime$ are \emph{isometric} if they can be connected
by a chain of isometries $G\ltimes M\xleftarrow{(\varphi,f)}K\ltimes Z\xrightarrow{(\omega,h)}H\ltimes N$.
If $O_{\textit{eff}}$ and $O_{\textit{eff}}^\prime$ are isometric, we say that $O$ and $O^\prime$
are \emph{effectively isometric}.
\end{definition}

For instance, any Riemannian manifold $M$ may be equipped with the trivial action of a finite group $G$
resulting in the noneffective orbifold $O$ presented by $G \ltimes M$.  The orbifolds $M$ and $O$
are identical in every sense that is significant to Riemannian geometry, and hence are effectively
isometric.  However, because the isotropy group of each point in an orbifold
is invariant under diffeomorphism, they are not isometric as orbifolds.

Every point in a connected, noneffective orbifold $O$ has nontrivial isotropy and hence is singular.
We refer to points that correspond to nonsingular points in the associated effective orbifold
$O_{\textit{eff}}$ as \emph{effectively nonsingular} and points that
correspond to singular points in the associated effective orbifold as \emph{effectively singular}.
An orbifold given by a smooth
manifold equipped with the trivial action of a finite group will be referred to as \emph{effectively smooth}.

% XXXXXXXXXXXXXXXXXXXXXXXXXXXXXXXXXXXXXXXXXXXXXXXXXXXXXXXXXXXXXX

\subsection{$\Gamma$-Sectors of an orbifold}
\label{subsec-BackSectors}

We recall the following.

\begin{definition}[\cite{farseagensectors}]
\label{def-GammaSectors} Let $\Gamma$ be a finitely generated
discrete group and let $O$ be presented as a quotient orbifold
$G\ltimes M$ as above.  Let $(\varphi)$ denote the $G$-conjugacy
class of a homomorphism $\varphi\co\Gamma\to G$.  The \emph{orbifold
of $\Gamma$-sectors} $\widetilde{O}_\Gamma$ of $O$ is the disjoint
union of orbifolds presented by
\[
    \bigsqcup\limits_{(\varphi) \in \mathrm{HOM}(\Gamma, G)/G} C_G(\varphi) \ltimes M^{\langle\varphi\rangle}
\]
where $M^{\langle\varphi\rangle}$ denotes the collection of points
fixed by each element of the image of $\varphi$ in $G$ and
$C_G(\varphi)$ denotes the centralizer of the image of $\varphi$.
For a given $\varphi \in \mathrm{HOM}(\Gamma,G)$, we refer to each
connected component of the orbifold presented by $C_G(\varphi)
\ltimes M^{\langle\varphi\rangle}$ as a \emph{$\Gamma$-sector} of
$O$.  We let $m_{(\varphi)}$ denote the number of connected components
of $C_G(\varphi) \ltimes M^{\langle\varphi\rangle}$ and let
$\widetilde{O}_{(\varphi)}^i$ for $i = 1, \ldots, m_{(\varphi)}$ denote
the corresponding connected orbifolds.  If
$C_G(\varphi) \ltimes M^{\langle\varphi\rangle}$ is connected, we denote
the corresponding sector simply $\widetilde{O}_{(\varphi)}$.
The sector corresponding to
the trivial homomorphism $\Gamma\to G$ is diffeomorphic to $O$ and
is referred to as the \emph{nontwisted sector} (or \emph{nontwisted
sectors} if $O$ is not connected); other sectors are referred to as
\emph{twisted sectors}.
\end{definition}

If $\Gamma = \Z$, since any homomorphism is completely
determined by its value on a generator of $\Z$, there is a
bijective correspondence between $\mathrm{HOM}(\Z, G)/G$ and the
conjugacy classes of $G$.  In this case, in order to simplify
notation, we identify the conjugacy class of a homomorphism $\Z\to
G$ with the conjugacy class of the image of a fixed generator of
$\Z$.  Similarly, $\mathrm{HOM}(\Z^\ell, G)/G$ corresponds to the
orbits of commuting $\ell$-tuples $(g_1, \ldots, g_\ell) \in G^\ell$
under the action of $G$ by simultaneous conjugation, while
$\mathrm{HOM}(\F_\ell, G)/G$, where $\F_\ell$ denotes the
free group with $\ell$ generators, corresponds to the orbits of (not
necessarily commuting) $\ell$-tuples under simultaneous conjugation.

It is easy to see that each $C_G(\varphi)$ acts with finite isotropy groups
on $M^{\langle\varphi\rangle}$ so that $C_G(\varphi) \ltimes
M^{\langle\varphi\rangle}$ does indeed present an orbifold. If $\psi = h\varphi h^{-1}$ for some $h \in G$, then
left-translation by $h$ induces a diffeomorphism between
$M^{\langle\varphi\rangle}$ and $M^{\langle\psi\rangle}$, and
conjugation by $h$ intertwines the respective actions of
$C_G(\varphi)$ and $C_G(\psi)$, so that the orbifold $C_G(\varphi)
\ltimes M^{\langle\varphi\rangle}$ does not depend on the choice of
representative $\varphi$ of $(\varphi)$.  Note
that $M^{\langle \varphi \rangle}$ is empty unless the image of
$\varphi$ is contained in the isotropy group of at least one point
$x \in M$. One implication is that if $O$ is a smooth manifold with
no group action, then $\widetilde{O}_\Gamma = O$ for each $\Gamma$.  More generally, if $\Gamma = \Z^\ell$ or $\Gamma = \F_\ell$ for $\ell \geq 1$, then
$\widetilde{O}_\Gamma = O$ if and only if $O$ is a manifold.

Suppose $O$ is a noneffective orbifold presented by $G\ltimes M$ so that $O_{\textit{eff}}$
is presented by $G/K \ltimes M$ where $K$ is the finite, normal subgroup
of $G$ acting trivially.  Let $\rho\co G \to G/K$ denote the quotient
homomorphism.  Then for each $\Gamma$, there is a surjective map
from $\mathrm{HOM}(\Gamma, G)/G$ to $\mathrm{HOM}(\Gamma,
G/K)/(G/K)$ given by sending $\varphi \in \mathrm{HOM}(\Gamma, G)$
to $\rho\circ\varphi \in \mathrm{HOM}(\Gamma, G/K)$. It is easy to
see that if $\varphi \in \mathrm{HOM}(\Gamma, G)$, then $M^{\langle
\varphi\rangle} = M^{\langle \rho\circ\varphi\rangle}$; however the
actions of $C_{G/K}(\rho\circ\varphi)$ on $M^{\langle
\rho\circ\varphi\rangle}$ and $C_G(\varphi)$ on
$M^{\langle\varphi\rangle}$ may differ.  Moreover, there are
generally more $\Gamma$-sectors of $O$ than $O_{\textit{eff}}$.  In
particular, if $O$ is a connected $n$-dimensional orbifold, only the
nontwisted sector of $O_{\textit{eff}}$ is $n$-dimensional, while each sector
of $O$ corresponding to a $\varphi \in \mathrm{HOM}(\Gamma,K)$ is
$n$-dimensional.

\begin{example}
\label{ex-noneffNewSectors} Let $\Sp^2$ denote the standard unit sphere, and
let $D_6 = \langle a, b : a^3 = b^2 = (ab)^2 = 1 \rangle$ denote the
dihedral group of order $6$.  Define a $D_6$-action on $\Sp^2$ where
$ax = x$ for each $x \in \Sp^2$, and $b$ acts as a rotation through
$\pi$ about a fixed axis.  Then $D_6 \ltimes \Sp^2$ presents a
noneffective orbifold $O$ with $K = \langle a \rangle$ acting
trivially.  The effectivization $O_{\textit{eff}}$ is presented by $\langle b \rangle
\ltimes \Sp^2$. The $\Z$-sectors of $O_{\textit{eff}}$ consist of
$\widetilde{(O_{\textit{eff}})}_{(1)}$, isometric to $O_{\textit{eff}}$, and
$\widetilde{(O_{\textit{eff}})}_{(b)}$, a pair of points with trivial
$\Z_2$-action. However, the $\Z$-sectors of $O$ consist of
$\widetilde{O}_{(1)}$, isometric to $O$, $\widetilde{O}_{(b)}$, a
pair of points with trivial $\Z_2$-action, and
$\widetilde{O}_{(a)}$, the standard unit sphere with trivial $\Z_3$-action.
\end{example}

More generally, if $O$ is presented by a Lie groupoid $\mathcal{G}$,
then the orbifold of $\Gamma$-sectors $\widetilde{O}_\Gamma$ can be
constructed as follows.  Let $\mathcal{S}_\Gamma^\mathcal{G}$ denote
the collection of groupoid homomorphisms $\mathrm{HOM}(\Gamma,
\mathcal{G})$ treating $\Gamma$ as a groupoid with a single object.
Then $\mathcal{S}_\Gamma^\mathcal{G}$ inherits from $\mathcal{G}$
the structure of a union of smooth manifolds (with connected components
of different dimensions) as well as a natural $\mathcal{G}$-action.
We let $\mathcal{G}^\Gamma$ denote the translation groupoid of
$\mathcal{G} \ltimes \mathcal{S}_\Gamma^\mathcal{G}$, and then
$\mathcal{G}^\Gamma$ is a presentation of $\widetilde{O}_\Gamma$.
See \cite{farseanonvanish} for the details of this construction.
Note that if $\mathcal{G} = G\ltimes M$ is a translation groupoid,
then a groupoid homomorphism $\varphi_x\co\Gamma \to \mathcal{G}$
corresponds to a choice of $x \in M$ and a group homomorphism
$\varphi\co\Gamma \to G_x \leq G$.  In particular, the
$\Gamma$-sector associated to $\varphi_x$ using the groupoid
definition is a connected orbifold; it corresponds to the connected
component of $\widetilde{O}_{(\varphi)}$ containing the orbit of
$x$. If $x \in M$ with isotropy group $G_x$ and
$\varphi\co\Gamma\to G_x \leq G$ is a homomorphism with image
contained in $G_x$, then a linear orbifold chart $\{ V_x, G_x, \pi_x
\}$ for $O$ at the orbit $Gx$ induces a linear orbifold chart
$\left\{ V_x^{\langle\varphi\rangle}, C_{G_x}(\varphi),
\pi_x^\varphi \right\}$ for $\widetilde{O}_{(\varphi)}$ at the point
$C_G(\varphi)x$.

\begin{proposition}[\cite{farseagensectors} and \cite{farseanonvanish}]
\label{prop-GammaSectorsIndepPres}
Let $\Gamma$ be a finitely generated discrete group.
A diffeomorphism of orbifolds $O \to O^\prime$ induces a diffeomorphism
$\widetilde{O}_\Gamma \to \widetilde{O^\prime}_\Gamma$ for each finitely
generated group $\Gamma$.  If $O$ is compact, then $\widetilde{O}_\Gamma$
is compact, and in particular consists of a finite number of connected components.
\end{proposition}

Hence, $\widetilde{O}_\Gamma$ does not depend on the presentation of
$O$.  The construction of $\widetilde{O}_\Gamma$ can be though of as
an ``unraveling" of the singular strata of $O$ into disjoint
orbifolds.  The choice of $\Gamma$ corresponds roughly with the
depth and type of singular strata that are unraveled.  For instance,
if $\Gamma = \Z$, then $\widetilde{O}_\Gamma$ is the inertia
orbifold, consisting of orbifolds that arise as fixed-point subsets
of cyclic groups in $O$.  If $\Gamma = \F_\ell$ is the free group
with $\ell$ generators, then $\widetilde{O}_\Gamma$ corresponds to
the $\ell$-multisectors; see \cite{ademleidaruan}.  Note that the
natural projection $\widetilde{O}_\Gamma \to O$ induced by the
inclusion of each $M^{\langle\varphi\rangle} \to M$ is not usually
injective, even when restricted to a single sector.

Given a metric $g$ on $O$, a local chart $\{ V_x, G_x, \pi_x \}$,
and a homomorphism $\varphi\co\Gamma\to G_x$, the metric $g_x$ on
$V_x$ restricts to a $C_{G_x}(\varphi_x)$-invariant metric on
$V_x^{\langle\varphi\rangle}$, inducing a metric on the associated
$\Gamma$-sector.  If $O$ is presented by $G \ltimes M$, a
choice of corresponding $G$-invariant metric on the smooth manifold
$M$ restricts to each $M^{\langle\varphi\rangle}$ as a
$C_G(\varphi)$-invariant metric, inducing an orbifold metric on each
$\Gamma$-sector as a quotient orbifold.
Using the slice theorem (see e.g. \cite[Theorem 2.3.3]{duistermaatKolk}),
it is easy to see that the restriction of the
$C_G(\varphi)$-invariant metric on $M^{\langle\varphi\rangle}$ to a chart
$V_x^{\langle\varphi\rangle}$ coincides with the restriction of the local metric
$g_x$ to $V_x^{\langle\varphi\rangle}$ so that the metric induced on
$\widetilde{O}_\Gamma$ depends only on the metric $g$ and not on the
presentation of $O$.  Similarly, given an isometry $O \to O^\prime$
between quotient orbifolds, the induced diffeomorphism
$\widetilde{O}_\Gamma \to \widetilde{O^\prime}_\Gamma$ preserves
the corresponding local metrics and hence is itself an isometry
of the corresponding $\Gamma$-sectors.

% XXXXXXXXXXXXXXXXXXXXXXXXXXXXXXXXXXXXXXXXXXXXXXXXXXXXXXXXXXXXXX

\subsection{$\Gamma$-Spectrum of an orbifold}
\label{subsec-BackSpectrum}

For a Riemannian orbifold $O$, we say that a function is smooth if at every point it can be lifted
to a smooth function on a local manifold cover above the point; equivalently, a smooth function is
an invariant smooth function on a presentation of $O$.  We denote the space of all smooth functions on $O$
by $C^{\infty}(O)$.   Since the Laplacian $\Delta$ acting on smooth function on a manifold commutes
with isometries, there is a well-defined action of the Laplacian on $C^{\infty}(O)$, computed by taking
the Laplacian of lifts of smooth functions to local manifold covers.  For any compact, connected
Riemannian orbifold $O$, the eigenvalue spectrum of the Laplacian, denoted $\mathrm{Spec}(O)$ is a
discrete sequence $0=\lambda_0<\lambda_1\leq \lambda_2\leq\cdots\uparrow \infty$, with each eigenvalue
appearing with finite multiplicity; see \cite{chiang}, \cite{dggw}.  We say that two orbifolds are
\textit{isospectral} if they have the same Laplace spectrum.

Let $\Gamma$ be a finitely generated discrete group and let $O$ be a compact orbifold presented by
$G \ltimes M$ where $G$ is a compact Lie group.  Then
\[
    C^\infty( \widetilde{O}_\Gamma )
    =
    \bigoplus\limits_{(\varphi) \in \mathrm{HOM}(\Gamma, G)/G} \;\; \bigoplus\limits_{i=1}^{m_{(\varphi)}}\;\;
    C^\infty(\widetilde{O}_{(\varphi)}^i).
\]
We let $\Delta_{(\varphi)}^i$ denote the corresponding Laplace
operator for each $\Gamma$-sector $\widetilde{O}_{(\varphi)}^i$ and
set
\[
    \Delta_{\Gamma}
    =
    \bigoplus\limits_{(\varphi) \in \mathrm{HOM}(\Gamma, G)/G} \;\; \bigoplus\limits_{i=1}^{m_{(\varphi)}}\;\;
    \Delta^i_{(\varphi)},
\]
so that $\Delta_\Gamma\co C^\infty(\widetilde{O}_\Gamma)\to C^\infty(\widetilde{O}_\Gamma)$.

\begin{definition}\label{defn-spectrum}
Let $O$ be a compact Riemannian orbifold and $\Gamma$ a finitely
generated discrete group. The \emph{$\Gamma$-spectrum of $O$} is the
spectrum of $\Delta_\Gamma$ acting on
$C^{\infty}(\widetilde{O}_\Gamma)$.  In other words
\begin{equation*}
    \mathrm{Spec}_\Gamma(O)=\bigcup\limits_{(\varphi)\in \mathrm{HOM}(\Gamma, G)/G}
    \quad \bigcup\limits_{i=1}^{m_{(\varphi)}}\mathrm{Spec}(\widetilde{O}_{(\varphi)}^i).
\end{equation*}
Two compact Riemannian orbifolds $O$ and $O^\prime$ are
\emph{$\Gamma$-isospectral} if $\mathrm{Spec}_\Gamma(O) =
\mathrm{Spec}_\Gamma(O^\prime)$.
\end{definition}

\begin{remark}
If $O$ is not presented by the quotient of a manifold $M$ by a group
$G$, then the $\Gamma$-spectrum of $O$ can be defined identically
by applying $\mathrm{Spec}$ to the $\Gamma$-sectors of $O$
defined using a groupoid presentation of $O$.  By
Proposition~\ref{prop-GammaSectorsIndepPres}, this definition coincides
with that given in Definition~\ref{defn-spectrum} when $O$ admits a presentation as a quotient.
\end{remark}

Hence, the $\Gamma$-spectrum of $O$ is the usual spectrum of the
$\Gamma$-sectors of $O$.  If $O$ is a connected Riemannian manifold,
then the only $\Gamma$-sector of $O$ is the nontwisted sector isometric to $O$
so the $\Gamma$-spectrum coincides with the usual notion of the Laplace
spectrum.  Similarly, if $\Gamma$ is the trivial group, then
$\mathrm{Spec}_\Gamma(O) = \mathrm{Spec}(O)$.
Note that the multiplicity of $0$ in the
$\Gamma$-spectrum of $O$ corresponds to the number of
$\Gamma$-sectors of $O$.

The usual Laplace spectrum of an orbifold $O$ coincides
with the Laplace spectrum of the associated effective orbifold
$O_{\textit{eff}}$.  This is not the case for the $\Gamma$-spectrum, as there
are generally more $\Gamma$-sectors of $O$ than $O_{\textit{eff}}$; see
Section~\ref{subsec-BackSectors} and in particular
Example~\ref{ex-noneffNewSectors}.  In addition, we consider the following.

\begin{example}
\label{ex-MtrivZ2} Let $M$ be any connected Riemannian manifold, and
let $O$ be presented by $\Z_2 \ltimes M$ where the $\Z_2$-action is
trivial. Then $\widetilde{M}_\Gamma \cong M$ for any $\Gamma$, while
$\widetilde{O}_\Z \cong O \sqcup O$.  It follows that
$\mathrm{Spec}_\Z(M) = \mathrm{Spec}(M)$ while for
$\mathrm{Spec}_\Z(O)$, the multiplicity of each eigenvalue from
$\mathrm{Spec}(M)$ is doubled. Hence, $M$ and $O$ are isospectral
and effectively isometric, though they are not $\Z$-isospectral.

More generally, if a connected orbifold $O$ is presented by $G \ltimes M$ for any finite group $G$ acting trivially,
then the number of connected components of $\widetilde{O}_\Z$ coincides with the number of
conjugacy classes in $G$.  Each connected component of $\widetilde{O}_\Z$ is effectively isometric to $O$,
though the group acting trivially may vary so that they need not be diffeomorphic.
\end{example}

\begin{example}
\label{ex-MtrivZ3D6} Let $M$ be any Riemannian manifold, let $O_1$
denote the quotient of $M$ by a trivial $\Z_3$-action, and let $O_2$
denote the quotient of $M$ by a trivial $D_6$-action where
$D_6 = \langle a, b\mid a^3 = b^2 = (ab)^2=1\rangle$
is the dihedral group with $6$ elements.  Then $O_1$ and $O_2$ are isospectral
and effectively isometric.  Moreover,
$\widetilde{(O_1)}_\Z$ is isometric to $O_1 \sqcup O_1 \sqcup O_1$, and since $D_6$
has $3$ conjugacy classes $(1)$, $(a)$, and $(b)$,
$\widetilde{(O_2)}_\Z \cong D_6\ltimes M \sqcup \Z_3\ltimes M \sqcup \Z_2\ltimes M$ with each
group action trivial.  Therefore, $O_1$ and $O_2$ are also
$\Z$-isospectral.  However, by counting the conjugacy classes of homomorphisms
$\Z^2\to \Z_3$ and $\Z^2\to D_6$, it is easy to see that
$\widetilde{(O_1)}_{\Z^2}$ has $9$ identical connected components
while $\widetilde{(O_2)}_{\Z^2}$ has $8$ connected components,
so that $O_1$ and $O_2$ are not $\Z^2$-isospectral.  Similarly,
$\widetilde{(O_1)}_{\F_2}$ has $9$ connected components
while $\widetilde{(O_2)}_{\F_2}$ has $12$, so that $O_1$ and $O_2$
are not $\F_2$-isospectral.
\end{example}

The above examples illustrate that consideration of noneffective
orbifolds yields trivial examples of orbifolds that, for instance,
are isospectral but not $\Z$-isospectral or are isospectral and
$\Z$-isospectral, but not $\Z^2$-isospectral.  In addition, consider
the following.

\begin{example}
\label{ex-GammaZpTriv} Let $O_1 = G_1 \ltimes M_1$ and $O_2 = G_2
\ltimes M_2$ be any pair of effective isospectral, non-isometric
orbifolds (see examples in Section~\ref{sec-RSW}), and let $p$ be a
prime that does not divide the order of the isotropy group of any
point in $O_1$ or $O_2$.  Then since every homomorphism $\Z_p \to G_i$,
$i = 1,2$ either has empty fixed point set or is trivial, it is easy
to see that $\widetilde{(O_1)}_{\Z_p} = O_1$ and
$\widetilde{(O_2)}_{\Z_p} = O_2$.  Therefore, $O_1$ and $O_2$ are
also $\Z_p$-isospectral.
\end{example}

Hence, many questions about the $\Gamma$-spectrum of a general
orbifold have trivial answers that involve algebraic trickery using
trivial group actions or a choice of $\Gamma$ that leads to no
nontwisted sectors.  We consider the $\Gamma$-spectrum to be
of most interest when applied to effective Riemannian orbifolds and
choices of $\Gamma$ that yield nontrivial sectors.

% XXXXXXXXXXXXXXXXXXXXXXXXXXXXXXXXXXXXXXXXXXXXXXXXXXXXXXXXXXXXXX

\subsection{Elementary $\Gamma$-spectral invariants}
\label{subsec-HeatKer}

Let $O$ be a compact, connected, effective Riemannian orbifold presented as
a quotient orbifold by $G\ltimes M$, and let
$0=\lambda_0<\lambda_1 \leq \lambda_2\leq\cdots\uparrow \infty$
denote the spectrum of $O$.  The \emph{heat trace} of $O$ is defined to be
$\sum_{j=1}^\infty e^{-\lambda_j t}$; see \cite{dggw} and \cite{donnelly}.
By \cite[Theorem 4.8]{dggw}, the heat trace as $t\to 0^+$ admits an asymptotic expansion of the form
\begin{equation}
\label{eq-DGGW4.8}
    (4\pi t)^{-\mathrm{dim}(O)/2}\sum\limits_{j=0}^\infty c_j t^{j/2}
\end{equation}
where $c_0 = \mathrm{vol}(O)$ is the Riemannian volume of $O$.
In particular, the volume and dimension of $O$ are determined by the spectrum; see
also \cite[Theorem 3.2]{farsi}.
It will be useful for us to recall that the asymptotic expansion of the heat trace
in Equation \ref{eq-DGGW4.8} can also be expressed as follows.  Let $\mathcal{S}(O)$
denote the strata of the singular set of $O$ with respect to its
Whitney stratification by orbit types.  Then the asymptotic expansion of the heat trace
can be decomposed into the contributions of the strata as
\begin{equation}
\label{eq-DGGW4.8byStrata}
    (4\pi t)^{-\mathrm{dim}(O)/2}\sum\limits_{k=0}^\infty a_k t^k
        + \sum\limits_{N\in\mathcal{S}(O)} (4\pi t)^{-\mathrm{dim}(N)/2}\sum\limits_{k=0}^\infty b_{k,N} t^k
\end{equation}
where the $a_k$ are the usual heat invariants as in the case of manifolds, and $b_{0,S} \neq 0$
for each $S$.  See \cite{donnelly}, \cite[Lemma 3.3]{GordonRossetti}, and
\cite[Definition 4.7 and Theorem 4.8]{dggw}.

In the case that $O$ is not effective, the heat trace of $O$ coincides with that
of $O_{\textit{eff}}$ since $\mathrm{Spec}(O)=\mathrm{Spec}(O_\textit{eff})$.
However, if $K$ denotes the isotropy group of an effectively
nonsingular point, then $\mathrm{vol}(O) = \mathrm{vol}(O_{\textit{eff}})/|K|$.
To see this, note that if $\{ V_x, G_x, \pi_x \}$ is an orbifold chart, a differential form $\omega$
on $\pi_x(V_x) \subseteq O$ can be defined locally as a $G_x$-invariant differential form on $V_x$,
and the integral of $\omega$ on $\pi_x(V_x)$ is defined to be
\[
    \int\limits_{\pi_x(V_x)} \omega :=
    \frac{1}{|G_x|} \int\limits_{V_x} \pi_x^\ast\omega,
\]
see \cite[p. 34]{ademleidaruan}.  If $p = \pi_x(x)$ is an effectively nonsingular
point in $O$, then $G_x \cong K$.  For an arbitrary point,
if $G_x^{\textit{eff}}$ denotes the isotropy group of the point
corresponding to $\pi_x(x)$ in $O_{\textit{eff}}$, then $|G_x^{\textit{eff}}| = |G_x|/|K|$
so that this integral differs from the
integral of the corresponding $\omega$ on $O_{\textit{eff}}$ by a factor of $|K|$.
Hence, if $O$ is not effective, then the volume can be determined from the spectrum along
with the order of the group $K$ acting trivially, but cannot be determined from the spectrum
alone.

For a finitely generated discrete group $\Gamma$, the \emph{$\Gamma$-heat trace of $O$}
is defined to be the heat trace of $\widetilde{O}_\Gamma$.  It is evidently
given by the sum of the heat traces of the sectors of $\widetilde{O}_\Gamma$.
Specifically, for each conjugacy class $(\varphi) \in \mathrm{HOM}(\Gamma,G)/G$
and each $i = 1, \ldots, m_{(\varphi)}$, let
$0=\lambda_0((\varphi),i)<\lambda_1((\varphi),i) \leq \lambda_2((\varphi),i)\leq\cdots\uparrow \infty$
denote the spectrum of the closed orbifold $\widetilde{O}_{(\varphi)}^i$,
the corresponding connected component of $C_G(\varphi)\ltimes M^{\langle\varphi\rangle}$.
Let $H_{(\varphi), i}(t) = \sum_{j=0}^\infty e^{-\lambda_j((\varphi),i) t}$ denote the
corresponding heat trace of $\widetilde{O}_{(\varphi)}^i$.  Then the $\Gamma$-heat trace
of $O$ is given by
\[
    \sum\limits_{(\varphi)\in \mathrm{HOM}(\Gamma,G)/G}
    \quad \sum\limits_{i=1}^{m_{(\varphi)}} H_{(\varphi), i}(t).
\]
Then the $\Gamma$-heat trace is asymptotic as $t\to 0^+$ to
\[
    \sum\limits_{(\varphi)\in \mathrm{HOM}(\Gamma,G)/G}
    \quad \sum\limits_{i=1}^{m_{(\varphi)}}
    (4\pi t)^{-\mathrm{dim}(\widetilde{O}_{(\varphi)}^i)/2}\sum\limits_{j=0}^\infty
    c_j((\varphi),i) t^{j/2},
\]
where the $c_j((\varphi),i)$ are the coefficients of the asymptotic expansion of the heat trace
of $\widetilde{O}_{(\varphi)}^i$.

If $O$ is effective, then the dimension of the nontwisted sector $\widetilde{O}_{(1)} \cong O$
is strictly larger than the dimension of each twisted sector.  Therefore, the lowest-degree
term in the asymptotic expansion of the heat trace is
$(4\pi t)^{-\mathrm{dim}(O)/2}\mathrm{vol}(O)$, with no contributions from the twisted
sectors.  In particular, the lowest degree term of the asymptotic expansions of the
$\Gamma$-heat trace and ordinary heat trace coincide.

Suppose on the other hand that $O$ is not effective so that a nontrivial finite subgroup
$K \unlhd G$ acts trivially on $M$.  Then the sectors $\widetilde{O}_{(\varphi)}$
corresponding to $\varphi$ with image contained in $K$ have dimension equal to
$\mathrm{dim}(O)$, while all other sectors have dimension strictly less
than $\mathrm{dim}(O)$.  Then as $G$ acts on $\mathrm{HOM}(\Gamma,K)$,
it follows that the lowest-degree term in the
asymptotic expansion of the heat trace is
\[
    (4\pi t)^{-\mathrm{dim}(O)/2}
    \sum\limits_{(\varphi)\in \mathrm{HOM}(\Gamma,K)/G}
    \mathrm{vol}\left( C_G(\varphi)\ltimes M \right),
\]
where the volumes in the sum need not be of connected orbifolds.

These observations yield the following.

\begin{proposition}
\label{prop-VolDimHTrace} Let $O$ be a compact, connected, effective
Riemannian orbifold. Then the volume and dimension of $O$ are
determined by the $\Gamma$-spectrum for any finitely generated
discrete group $\Gamma$.  If $O$ is not effective, then the
dimension of $O$ is determined by the $\Gamma$-spectrum for any
finitely generated discrete group $\Gamma$.
\end{proposition}

Note that as in the case of the ordinary spectrum, the asymptotic
expansion of the heat trace is a strictly coarser invariant than the
spectrum itself.  See Section~\ref{subsubsec-3.7}
for an example of orbifolds for which the
asymptotic expansions of the $\Gamma$-heat traces coincide for every
group $\Gamma$ though the orbifolds are not $\Gamma$-isospectral for every
$\Gamma$.

\bigskip
% XXXXXXXXXXXXXXXXXXXXXXXXXXXXXXXXXXXXXXXXXXXXXXXXXXXXXXXXXXXXXX
% XXXXXXXXXX    SECTION: KNOWN EXAMPLES
% XXXXXXXXXXXXXXXXXXXXXXXXXXXXXXXXXXXXXXXXXXXXXXXXXXXXXXXXXXXXXX

\section{Isospectral, nonisometic orbifolds and their $\Gamma$-spectra}
\label{sec-RSW}

In this section, we consider the $\Gamma$-spectra of examples of
isospectral, nonisometric orbifolds that have been given in the
literature.  Along with using these examples to illustrate features
of the $\Gamma$-spectrum, we will see that in many examples, the
$\Z$-spectrum is able to distinguish between isospectral pairs;
this is the case for the examples recalled in Sections
\ref{subsec-ssw} through \ref{subsec-rswFlat}.
We will also indicate known examples of nonisometric
orbifolds with nontrivial singular set that are $\Gamma$-isospectral
for every choice of $\Gamma$ in Section~\ref{subsec-shams}.
Throughout this section, we restrict our attention to pairs of
orbifolds with nontrivial singular sets because if $O$ and $O^\prime$ are isospectral
manifolds, then they are automatically $\Gamma$-isospectral for every choice of
$\Gamma$.

% XXXXXXXXXXXXXXXXXXXXXXXXXXXXXXXXXXXXXXXXXXXXXXXXXXXXXXXXXXXXXX

\subsection{Examples of Shams-Stanhope-Webb}
\label{subsec-ssw}

In \cite{ssw}, given an odd prime $p$ and an integer $m\geq 1$, Shams, Stanhope,
and Webb construct a family $\{ O_i : i = 0, \ldots, m\}$ of pairwise isospectral,
nonisometric orbifolds.  As we will see below, no $O_i$ and $O_j$ with $i \neq j$
are $\Z$-isospectral.  These examples illustrate how one can conclude
that orbifolds are not $\Z$-isospectral without determining the $\Z$-sectors
explicitly.

The orbifolds $O_i$ are given by quotients of the
standard unit sphere $\Sp^{p^{3m}-1}$ in $\R^{p^{3m}}$ by subgroups of the permutation group $S_{p^{3m}}$
acting on a basis for $\R^{p^{3m}}$.   Specifically, let $H$ denote the mod-$p$ Heisenberg group and let
$E = (\Z_p)^3$. Then $O_i$ is given by $H_i\ltimes \Sp^{p^{3m}-1}$ where $H_i = H^i\times E^{m-i}$,
realized as a subgroup of $S_{p^{3m}}$.

To see that no pair $O_i$ and $O_j$ are $\Z$-isospectral for $i\neq
j$, we first consider the case $m = 1$, i.e. the two orbifolds $O_0
= E \ltimes \Sp^{p^3-1}$ and $O_1 = H \ltimes \Sp^{p^3-1}$. To
compute the fixed-point sets, note that each nontrivial element $a$
of $E$ has order $p$ so that the left action of $a$ partitions $E$
into $p^2$ orbits of size $p$. Given a standard basis vector $e_i$
of $\R^{p^3}$, let $e_i^a = \sum_{k=0}^{p-1} a^k e_i$ denote the
average of $e_i$ over the action of $\langle a \rangle$.  Note that
$e_i^a = e_j^a$ if and only if $e_i$ and $e_j$ are in the same
orbit, and hence there are $p^2$ distinct $e_i^a$ corresponding to
the $p^2$ $\langle a\rangle$-orbits in $E$.  A vector in $\R^{p^3}$
is fixed by $a$ if and only if it is a linear combination of the
$e_i^a$, so that $(\R^{p^3})^a$ is a subspace of dimension $p^2$.  Then
$(\Sp^{p^3-1})^a$ is $(p^2 - 1)$-dimensional, and in particular is a
subsphere of $\Sp^{p^3-1}$ of positive dimension, hence connected.
As $E$ and $H$ are almost conjugate, the same holds for true for the
nontrivial elements of $H$.

Since $E$ is abelian, the $\Z$-sectors $\widetilde{(O_0)}_\Z$
consist of the nontwisted sector $O_0$ as well as $p^3 - 1$ twisted
sectors of the form $E \ltimes \Sp^{p^2-1}$. Thus,
$\widetilde{(O_0)}_\Z$ has $p^3$ connected components. On the other
hand, since $H$ is not abelian, $H$ contains strictly fewer than
$p^3$ conjugacy classes and therefore $\widetilde{(O_1)}_\Z$ has
strictly fewer than $p^3$ connected components, each of which can be
represented by a quotient of $\Sp^{p^2-1}$ by a subgroup of $H$.  We
conclude that $O_0$ and $O_1$ are not $\Z$-isospectral since the
multiplicities of $0$ in their $\Z$-spectra do not coincide.

For general $m$, note that the bijection from $\mathrm{HOM}(\Z, A\times B)$
to $\mathrm{HOM}(\Z,A)\times\mathrm{HOM}(\Z,B)$ given by
$\varphi \mapsto (\pi_A\circ\varphi,\pi_B\circ\varphi)$, where $\pi_A$
and $\pi_B$ denote the respective projection homomorphisms, is equivariant with
respect to conjugation by $A\times B$ and hence induces a bijection between
$\mathrm{HOM}(\Z, A\times B)/(A\times B)$
and $\mathrm{HOM}(\Z,A)/A\times\mathrm{HOM}(\Z,B)/B$.  Hence, the above
argument demonstrates that for $i < j$, the $\Z$-sectors of
$H_i \ltimes \Sp^{p^{3m}-1}$ have strictly fewer connected components than the
$\Z$-sectors of $H_j \ltimes \Sp^{p^{3m}-1}$.
Thus, we conclude that no pair of these orbifolds is $\Z$-isospectral.

% XXXXXXXXXXXXXXXXXXXXXXXXXXXXXXXXXXXXXXXXXXXXXXXXXXXXXXXXXXXXXX

\subsection{Homogeneous space examples of Rossetti-Schueth-Weilandt}
\label{subsec-rswHS}

In \cite{rsw}, Rossetti, Schueth, and Weilandt describe several pairs of isospectral, nonisometric orbifolds, demonstrating in particular that isospectral orbifolds need not have the same maximal isotropy order.  The first three examples that they give (Examples 2.7--9) are biquotients of $SO(6)$.  In each of these
examples, the resulting isospectral orbifold pairs are not $\Z^\ell$- or $\F_\ell$-isospectral for any
$\ell \geq 1$.  We will describe this computation for their Examples 2.7 and 2.9 explicitly in Sections
\ref{subsubsec-rswSO3} and \ref{subsubsec-rswSO5} to illustrate the computation of the $\Gamma$-sectors
and $\Gamma$-spectrum.
Note that the computation for their Example 2.8 is similar to that carried out for Example 2.7 below;
the resulting twisted sectors of the two orbifolds have a common Riemannian cover and can be seen to be
not isospectral by an application of \cite[Proposition 3.4(ii)]{GordonRossetti}.

In order to describe these examples, let $a_{i_1i_2\cdots i_m}$ to denote a square diagonal $6\times 6$
matrix with $-1$ in the $i_1, i_2, \dots, i_m$ positions and $1$ everywhere else.  We recall from \cite[Example 2.4]{rsw} the groups
\begin{equation}
\label{eq-RSWGp1}
    K_1 = \{I, -I, a_{12}, a_{13}, a_{23}, a_{1456}, a_{2456}, a_{3456}\}
\end{equation}
and
\begin{equation}
\label{eq-RSWGp2}
    K_2=\{I, -I, a_{12}, a_{34}, a_{56}, a_{1234}, a_{1256}, a_{3456}\}.
\end{equation}

The orbifold pairs in Examples 2.7-9 in \cite{rsw} are of the form $O_1=K_1\ltimes (SO(6)/H)$ and $O_2=K_2\ltimes (SO(6)/H)$ where $H$ is a choice of subgroup in $SO(6)$.  They are isospectral by Sunada's theorem.  (See Section~\ref{sec-Sunada} below for a discussion of Sunada's theorem.)

%XXXXXXXXXXXXXXXXXXXXXXXX

\subsubsection{$H \cong SO(3)$, \cite[Example 2.7]{rsw}}
\label{subsubsec-rswSO3}
In this case, $H \cong SO(3)$ is chosen to be the subgroup of matrices in $3\times 3$
blocks of the form
\[
    \left[\begin{array}{cc}
        A           &   0       \\
        0           &   I_3
    \end{array}\right]
\]
with $A\in SO(3)$, so that $M = SO(6)/H$ is the Steifel manifold $V_{6,3}$ of
orthonormal $3$-frames in $\R^6$.  In particular, for $b \in SO(6)$,
the coset $bH$ corresponds to the frame given by the last three
columns of $b$. With a choice of biinvariant metric on $SO(6)$, the
orbifolds $O_1 = K_1 \ltimes M$ and $O_2 = K_2 \ltimes M$ are isospectral with different maximal isotropy orders.

To compute the $\Z$-sectors of $O_1$ and $O_2$, note that an element
$bH \in SO(6)/H$ is fixed by $a_{i_1i_2\cdots i_m}$ (acting on the left) if and only if
the last three columns of $b$ have the zero rows in positions $i_1,i_2,\dots, i_m$.  It is easy to
see that
\[
    M^{a_{1456}} = M^{a_{2456}} = M^{a_{3456}} =
    M^{a_{1234}} = M^{a_{1256}} = M^{-I} = \emptyset,
\]
because if $bH$ were fixed by any of these elements of $K_1$ or $K_2$,
then the last three columns of $b$ would be a linearly independent set
of three vectors in a subspace of dimension zero or two.
Thus the $\Z$-sectors of $O_1$ and $O_2$ both consist of
four connected components.

In the case of $O_1$, besides the nontwisted sector, the three
$\Z$-sectors
\[
    \widetilde{(O_1)}_{(a_{12})} \cong \widetilde{(O_1)}_{(a_{13})} \cong \widetilde{(O_1)}_{(a_{23})}
\]
are isometric; we describe $\widetilde{(O_1)}_{(a_{12})}$ in detail.
The fixed point set $M^{a_{12}}$ consists of cosets $bH$ corresponding to elements $b \in SO(6)$
of the form
\[
    \left[\begin{array}{cc}
        \ast_{2\times 3}    &   0_{2\times 3}       \\
        \ast_{4\times 3}    &   v_{4\times 3}
    \end{array}\right].
\]
Here, $v_{4\times 3}$ is an orthonormal $3$-frame in $\R^4$ that depends only on the coset $bH$,
while $\ast$ indicates entries of $b$ that depend on the choice of representative from $bH$.
Since $K_1$ is abelian, the centralizer of any element is the
entire group, so $\widetilde{(O_1)}_{(a_{12})}\cong K_1\ltimes
M^{a_{12}}$.

In order to understand the action of $K_1$ on $M^{a_{12}}$, we first note that $K_1\cong (\Z_2)^3 \cong \langle a_{12}\rangle\oplus\langle a_{13}\rangle\oplus \langle a_{1456}\rangle$.   Thus the action of $K_1$ corresponds to a trivial $\Z_2$-action generated by $a_{12}$ as well as a nontrivial $\Z_2\oplus\Z_2$-action generated by $a_{13}$ and $a_{1456}$ on the
rows indexed $3456$ of $v_{4\times 3}$.  The action of
$a_{13}$ fixes the set of cosets for which the first row of
$v_{4\times 3}$ vanishes, and therefore the fixed point set of $a_{13}$ is isometric to $SO(3)$.  All other elements of $\Z_2\oplus\Z_2$ have no fixed points because, as above, it is impossible to have three linearly independent vectors in a subspace of dimension less than three.

The remaining twisted sectors
$\widetilde{(O_1)}_{(a_{13})}$ and $\widetilde{(O_1)}_{(a_{23})}$
are identical up to permuting rows.

In the case of $O_2$, the three twisted $\Z$-sectors are
\[
    \widetilde{(O_2)}_{(a_{12})} \cong \widetilde{(O_2)}_{(a_{34})} \cong \widetilde{(O_2)}_{(a_{56})}
\]
and are again isometric simply by permuting rows, so we focus our attention on  $\widetilde{(O_2)}_{(a_{12})}$.  The fixed point set $M^{a_{12}}$ is as in the
case of $O_1$, and again, since $K_2$ is abelian, $\widetilde{(O_2)}_{(a_{12})}\cong K_2\ltimes M^{a_{12}}$.  In this case, $K_2\cong (\Z_2)^3\cong \langle a_{12}\rangle\oplus\langle a_{34}\rangle\oplus\langle a_{56}\rangle$ so the action of $K_2$ corresponds to a
trivial $\Z_2$-action generated by $a_{12}$ as well as a $\Z_2\oplus\Z_2$-action generated
by $a_{34}$ and $a_{56}$ on the rows indexed $3456$ of $v_{4\times
3}$.  As above, since there cannot be three linearly independent vectors in a subspace of dimension less than three, the
$\Z_2\oplus\Z_2$-action is free.  We conclude therefore that
$\widetilde{(O_2)}_{(a_{12})}$ (and thus $\widetilde{(O_2)}_{(a_{34})}$ and $\widetilde{(O_2)}_{(a_{56})}$)  is a smooth manifold.

To see that $O_1$ and $O_2$ are not $\Z$-isospectral, note that the
$\Z$-spectrum of each $O_i$ is the union
\[
    \mathrm{Spec}(O_i) \cup 3\mathrm{Spec}\left(\widetilde{(O_i)}_{(a_{12})}\right),
\]
where the $3$ indicates that the multiplicity of each element of the spectrum
is multiplied by $3$. Since $O_1$ and $O_2$ are isospectral, it is sufficient to show that
$\widetilde{(O_1)}_{(a_{12})}$ and $\widetilde{(O_2)}_{(a_{12})}$
are not isospectral.  The effective orbifold associated
to $\widetilde{(O_1)}_{(a_{12})}$ is a $6$-dimensional orbifold with
$3$-dimensional singular set, while the effective orbifold
associated to $\widetilde{(O_2)}_{(a_{12})}$ is a smooth
$6$-dimensional manifold.  Hence, it follows from \cite[Theorem 5.1]{dggw}
that they are not isospectral.

It is also of interest to consider the $\Gamma$-sectors of the
orbifolds $O_1$ and $O_2$ for other free groups $\Gamma$.  Since
$K_1$ and $K_2$ are abelian, the
$\Z^2$-sectors and the $\F_2$-sectors coincide; see \cite{dsst2orb}.
The fact that the $\Z^2$-sectors coincide with the
$\Z$-sectors of the $\Z$-sectors computed above (see \cite[Theorem
3.1]{farseageneuler}) makes it straightforward to compute the $\Z^2$-sectors of
$O_1$ and $O_2$.

For $O_1$, from the nontwisted $\Z$-sector $O_1$, a computation identical to the one above gives four
$\Z^2$-sectors: one copy of $O_1$ and three sectors that are isometric to
$\widetilde{(O_1)}_{(a_{12})}\cong K_1\ltimes M^{a_{12}}$.  From each of the three twisted $\Z$-sectors
$\widetilde{(O_1)}_{(a_{12})}$ we get two copies of $\widetilde{(O_1)}_{(a_{12})}$ (correponding to
homomorphisms $\Z\to K_1$ with image $I$ and $\langle a_{12}\rangle$ respectively) and two copies of
$K_1\ltimes SO(3)$ (corresponding to homomorphisms $\Z\to K_1$ with images $\langle a_{13}\rangle$ and $\langle a_{23}\rangle$ respectively).  Thus in total, the $\Z^2$-sectors of $O_1$ are given by

\begin{itemize}
\item   the nontwisted sector isometric to $O_1$;
\item   nine isometric copies of $\widetilde{(O_1)}_{(a_{12})}$; and
\item   six isometric copies of $K_1  \ltimes SO(3)$.
\end{itemize}

Similarly, for $O_2$, from the nontwisted $\Z$-sector $O_2$ we obtain one copy of $O_2$ and three sectors that are isometric to $\widetilde{(O_2)}_{(a_{12})}\cong K_2\ltimes M^{a_{12}}$.  From each of the three twisted $\Z$-sectors $\widetilde{(O_2)}_{(a_{12})}$ we get two copies of $\widetilde{(O_2)}_{(a_{12})}$, corresponding to homomorphisms $\Z\to K_2$ with image $I$ and $\langle a_{12}\rangle$ respectively. Thus the $\Z^2$-sectors of $O_2$ are

\begin{itemize}
\item   the nontwisted sector isometric to $O_2$; and
\item   nine isometric copies of $\widetilde{(O_2)}_{(a_{12})}$.
\end{itemize}

Since the number of connected components of the $\Z^2$-sectors of $O_1$ and $O_2$ respectively do not
coincide, the multiplicities of $0$ in the $\Z^2$-spectra of $O_1$ and $O_2$ do not
coincide, thus $O_1$ and $O_2$ are not
$\Z^2$-isospectral.

As the isotropy group of each point in $O_1$ and $O_2$ is abelian
and can be generated by two elements, the $\Z^\ell$-sectors of each
$O_i$ for $\ell > 2$ will simply yield multiple copies of the
$\Z^2$-sectors of $O_i$.  In general, by counting homomorphisms $\Z^{\ell}\to K_i$ whose images have nontrivial fixed point sets, we conclude that $\widetilde{(O_1)}_{\Z^\ell}$
consists of $4^\ell$ connected components while
$\widetilde{(O_2)}_{\Z^\ell}$ consists of $3 \cdot 2^\ell - 2$
connected components, so these orbifolds are not
$\Z^\ell$-isospectral for any positive $\ell$.

% XXXXXXXXXXXXXXXXXXXXXXXX

\subsubsection{$H \cong SO(5)$ \cite[Example 2.9]{rsw}}
\label{subsubsec-rswSO5}
In this case, $H \cong SO(5)$ is chosen to be the subgroup of matrices of the form
\[
    \left[\begin{array}{cc}
        A           &   0       \\
        0           &   1
    \end{array}\right]
\]
with $A \in SO(5)$ so that $M = SO(6)/H$ is isometric to the
standard sphere $\Sp^5$ in $\R^6$. The twisted $\Z$-sectors of the
orbifold $O_1 = K_1\ltimes M$ are given by
\[
    \widetilde{(O_1)}_{(a_{12})} \cong \widetilde{(O_1)}_{(a_{13})} \cong \widetilde{(O_1)}_{(a_{23})},
\]
each isometric to the standard sphere $\Sp^3$ with trivial
$\Z_2$-action generated by $a_{12}$ and $\Z_2\oplus\Z_2$-action generated by $a_{13}$ and
$a_{1456}$ in coordinates indexed $3456$ for $\R^4$, as well as
\[
    \widetilde{(O_1)}_{(a_{1456})} \cong \widetilde{(O_1)}_{(a_{2456})} \cong \widetilde{(O_1)}_{(a_{3456})},
\]
each isometric to $\Sp^1$ with trivial $\Z_2$-action generated by $a_{1456}$ and
$\Z_2\oplus\Z_2$-action generated by $a_{12}$ and $a_{13}$ in coordinates
indexed $23$ for $\R^2$.

Similarly, the twisted $\Z$-sectors of $O_2 = K_2\ltimes M$ are
given by
\[
    \widetilde{(O_2)}_{(a_{12})} \cong \widetilde{(O_2)}_{(a_{34})} \cong \widetilde{(O_2)}_{(a_{56})},
\]
each isometric to the standard sphere $\Sp^3$ with trivial
$\Z_2$-action generated by $a_{12}$ and $\Z_2\oplus\Z_2$-action generated by $a_{34}$ and
$a_{56}$ in coordinates indexed $3456$ for $\R^4$, as well as
\[
    \widetilde{(O_2)}_{(a_{1234})} \cong \widetilde{(O_2)}_{(a_{1256})} \cong \widetilde{(O_2)}_{(a_{3456})},
\]
each isometric to $\Sp^1$ with trivial $\Z_2\oplus\Z_2$-action generated by $a_{12}$ and $a_{34}$, and
$\Z_2$-action generated by $a_{56}$.

It is possible to compute the small values of the
$\Z$-spectrum directly using the fact that the eigenfunctions of the
Laplacian on a standard sphere are given by the restrictions of the
homogeneous harmonic polynomials on $\R^n$; see \cite{bgm}.  It follows
that the eigenfunctions on an orbifold space form are given by the
invariant homogeneous harmonic polynomials; see \cite{rsw}.  By computing bases for the $k^{th}$ eigenspaces of
$\Sp^1$ and $\Sp^3$ and checking invariance directly, one computes
that the first elements of the spectrum of
$\widetilde{(O_1)}_{(a_{1456})}$ are $0$ and $4$, both with a
multiplicity of $1$.  The next eigenvalue of  $\widetilde{(O_1)}_{(a_{1456})}$ must be at least 9.  The first eigenvalue of
$\widetilde{(O_1)}_{(a_{12})}$ is $0$ with a multiplicity of $1$ and
the next eigenvalue is $8$.

On the other hand, the first
elements of the spectrum of $\widetilde{(O_2)}_{(a_{1234})}$ are $0$
with a multiplicity of $1$ and $4$ with a multiplicity of $2$, and the next eigenvalue is at least 9.
The first eigenvalue of $\widetilde{(O_2)}_{(a_{12})}$ is $0$
with a multiplicity of $1$, and the next eigenvalue is $8$.

It
follows that  $\Z$-spectrum of $O_1$ is given by
\[
    \mathrm{Spec}(O_1) \cup \{ 0_6, 4_3, \cdots \}
\]
with subscripts indicating multiplicity, while the $\Z$-spectrum of
$O_2$ is given by
\[
    \mathrm{Spec}(O_2) \cup \{ 0_6, 4_6, \cdots \}.
\]
Since $\mathrm{Spec}(O_1) = \mathrm{Spec}(O_2)$, the multiplicity of $4$ cannot coincide in the $\Z$-spectra of $O_1$ and $O_2$, and hence they
are not $\Z$-isospectral.

Similarly, because $\widetilde{(O_1)}_{(a_{12})}$, $\widetilde{(O_1)}_{(a_{1456})}$, and
$\widetilde{(O_2)}_{(a_{12})}$ all have $6$ $\Z$-sectors, while
$\widetilde{(O_2)}_{(a_{1234})}$ has $4$ $\Z$-sectors, it is easy to see that
$\widetilde{(O_1)}_{\Z^\ell} = \widetilde{(O_1)}_{\F_\ell}$ has more connected
components than $\widetilde{(O_2)}_{\Z^\ell} = \widetilde{(O_2)}_{\F_\ell}$, so that
$O_1$ and $O_2$ are not $\Z^\ell$- or $\F_\ell$-isospectral for any $\ell$.

% XXXXXXXXXXXXXXXXXXXXXXXXXXXXXXXXXXXXXXXXXXXXXXXXXXXXXXXXXXXXXX

\subsection{Flat space examples of Rossetti-Schueth-Weilandt}
\label{subsec-rswFlat} In addition to the examples given above,
Rossetti, Scheuth, and Weilandt also describe pairs of isospectral,
nonisometric orbifolds given by quotients of $\R^3$ with its
standard, flat metric by pairs of crystallographic groups $K_1$ and
$K_2$, i.e. groups of isometries of $\R^3$ that act properly
discontinuously with compact quotients.  In these cases, the
resulting orbifolds are shown to be isospectral using either Sunada's theorem or an
eigenspace dimension counting formula \cite[Theorem 3.1]{rsw}; see
also \cite{MiatRos1, MiatRos2}. In every case, the resulting
orbifolds are not $\Z^\ell$- or $\F_\ell$-isospectral for $\ell \geq 1$.
This follows from the fact that the collections of twisted sectors are not
isospectral, similar to the examples in Section~\ref{subsec-rswHS}.
However, because the singular sets of these orbifolds are described in detail in
\cite{rsw}, the sectors can be computed directly from these
descriptions using Remark \ref{rem-SingSetExplicit} below.
We will briefly describe the sectors and consequence
for Examples 3.3, 3.5, and 3.7 of \cite{rsw} in order to illustrate
this approach; note that Examples 3.9 and 3.10 can be treated identically.
In each case, $O_i = K_i\ltimes \R^3$ for $i = 1, 2$.  See \cite{weilandt} for more details.

\begin{remark}
\label{rem-SingSetExplicit}
Let $O$ be a quotient orbifold represented by $G \ltimes M$, $\Gamma$ a finitely generated discrete group,
and $\varphi\co\Gamma\to G$ a homomorphism such that $M^{\langle\varphi\rangle} \neq \emptyset$.
Recall that a linear orbifold chart $\{ V_x, G_x, \pi_x \}$ for $O$ at the orbit $Gx$ induces a chart
$\left\{ V_x^{\langle\varphi\rangle}, C_{G_x}(\varphi), \pi_x^\varphi \right\}$ for $\widetilde{O}_{(\varphi)}$ at the point
$C_G(\varphi)x$.  Hence, the $\Gamma$-sectors can be determined locally in terms of an orbifold chart and then patched
together, which is often convenient when the singular set and isotropy groups of $O$ are known explicitly.
We use this fact when computing sectors in each of the examples in this section.
\end{remark}

% XXXXXXXXXXXXXXXXXXXXXXXX

\subsubsection{\cite[Example 3.3]{rsw}}
\label{subsubsec-3.3}
In this example, the singular sets of the orbifolds $O_1$ and $O_2$ each
consist of a disjoint union of circles.  For each $\Gamma = \Z^\ell$
or $\F_\ell$ with $\ell \geq 1$, the $\Gamma$-sectors of $O_1$ and $O_2$
have a different number of connected components so that $O_1$ and $O_2$
are not $\Gamma$-isospectral.

The singular set of $O_1$ consists of three
circles $\Sp^1$ of length $1$ with isotropy groups $\Z_4$, $\Z_4$,
and $\Z_2$, respectively.  A linear orbifold chart for a point contained in a singular circle
with isotropy group $\Z_k$ is of the form $\{ V_x, G_x, \pi_x \}$ where $V_x$
is diffeomorphic to $\R^3$ and $G_x \cong \Z_k$ acts as rotations about an axis.
Hence the corresponding charts for $\widetilde{(O_1)}_\Z$ are parameterized
by homomorphisms $\varphi\co\Z\to\Z_k$, where the trivial homomorphism
yields a chart for the nontwisted sector and each nontrivial homomorphism
yields a chart of the form $\left\{ V_x^{\langle\varphi\rangle}, C_{G_x}(\varphi), \pi_x^\varphi \right\}$
where $V_x^{\langle\varphi\rangle}$ is a line on which $C_{G_x}(\varphi) \cong \Z_k$ acts trivially.
These charts patch together to describe a neighborhood of the circle in the nontwisted sector as well as
one circle with trivial $\Z_k$-action for each nontrivial homomorphism $\varphi\co\Z\to\Z_k$.

It follows that the twisted $\Z$-sectors
of $O_1$ consist of seven copies of $\Sp^1$ with length $1$, six with
trivial $\Z_4$-action and one with trivial $\Z_2$-action.  The
singular set of $O_2$, on the other hand, consists of four copies of
$\Sp^1$, each with $\Z_2$-isotropy, two of length $2$ and two of
length $1$.  Therefore, the twisted $\Z$-sectors of $O_2$ consist of
four copies of $\Sp^1$ with trivial $\Z_2$-action in pairs of length
$2$ and $1$.  As the numbers of connected components do not
coincide, $O_1$ and $O_2$ are not $\Z$-isospectral.  Similarly, for
$\ell > 1$, $O_1$ and $O_2$ are easily seen to not be $\Z^\ell$- or
$\F_\ell$-isospectral by counting numbers of nontrivial homomorphisms
from $\Z^{\ell}$ into respective isotropy groups.  Indeed, $\widetilde{(O_1)}_{\Z^\ell} =
\widetilde{(O_1)}_{\F_\ell}$ has $2(4^\ell - 1) + 2^\ell$ connected
components, each twisted sector a circle of length $1$ with
$\Z_2$- or $\Z_4$-isotropy, while $\widetilde{(O_2)}_{\Z^\ell} =
\widetilde{(O_2)}_{\F_\ell}$ has $4\cdot 2^\ell - 3$ connected
components, each twisted sector a circle of length $1$ or $2$
with $\Z_2$-isotropy.

% XXXXXXXXXXXXXXXXXXXXXXXX

\subsubsection{\cite[Example 3.5]{rsw}}
\label{subsubsec-3.5}
This example is similar to that treated in Section \ref{subsubsec-3.3}
above, though the singular set of
$O_2$ is more interesting.  We again have that the $\Gamma$-sectors
of $O_1$ and $O_2$ have a different number of connected components when
$\Gamma = \Z^\ell$ or $\F_\ell$ with $\ell \geq 1$.

The singular set of $O_1$ consists of three
circles $\Sp^1$ of length $2$ with isotropy groups $\Z_4$, $\Z_4$,
and $\Z_2$, respectively.  It follows that the twisted $\Z$-sectors
of $O_1$ consist of $7$ copies of $\Sp^1$ with length $2$, six with
trivial $\Z_4$-action and one with trivial $\Z_2$-action. The
orbifold $O_2$ has singular set given by a trivalent graph with $8$
vertices and $12$ edges forming the $1$-skeleton of a cube, where each vertex
has $\Z_2\times\Z_2$-isotropy and each edge has $\Z_2$-isotropy; the
action of each $\Z_2$ on $\R^3$ is given by a rotation through $\pi$
about a single axis, while the action of each $\Z_2\times\Z_2$ is
given by rotation through $\pi$ about two orthogonal axes. Because
the vertices do not have cyclic isotropy, they do not appear as
$0$-dimensional $\Z$-sectors.  Rather, the twisted $\Z$-sectors are
twelve mirrored intervals of length $1$ with $\Z_2$-isotropy on the
interior and $\Z_2\times\Z_2$-isotropy on the endpoints. That is,
each twisted sector is the quotient of a circle $\Sp^1$ of length
$2$ by a $\Z_2\times\Z_2$-action, where one $\Z_2$-factor acts
trivially and the other acts by reflection through a diameter.
Since the numbers of connected components do not match, $O_1$ and $O_2$ are not $\Z$-isospectral.
Similarly, for $\ell > 1$, $O_1$ and $O_2$ are easily seen to not be
$\Z^\ell$- or $\F_\ell$-isospectral. Indeed,
$\widetilde{(O_1)}_{\Z^\ell} = \widetilde{(O_1)}_{\F_\ell}$ has
$2(4^\ell - 1) + 2^\ell$ connected components, each twisted sector a
circle of length $2$ with $\Z_2$- or $\Z_4$-isotropy.  On the other
hand, $\widetilde{(O_2)}_{\Z^\ell} = \widetilde{(O_2)}_{\F_\ell}$
has $2^{2\ell+3} - 3\cdot 2^{\ell+2} + 5$ connected components
consisting of $8(4^\ell - 3\cdot 2^\ell + 2)$ points with
$\Z_2\times\Z_2$-isotropy, $12(2^\ell-1)$ mirrored intervals with
$\Z_2$-isotropy on the interior, and the nontwisted sector.

% XXXXXXXXXXXXXXXXXXXXXXXX

\subsubsection{\cite[Example 3.7]{rsw}}
\label{subsubsec-3.7}
The orbifolds in this example are again similar to
those in Section \ref{subsubsec-3.3} above, and are not
$\Gamma$-isospectral for any $\Gamma$ that admits a nontrivial
homomorphism to $\Z_2$.  However, it is of interest to note
that the asymptotic expansions of the $\Gamma$-heat kernels of $O_1$ and $O_2$ coincide
for every group $\Gamma$; see Section~\ref{subsec-HeatKer}.

The orbifold $O_1$ in this case has a singular set
consisting of two copies of $\Sp^1$ of length $\sqrt{2}$ with
$\Z_2$-isotropy.  Hence the twisted $\Z$-sectors consist of
two copies of $\Sp^1$ of length $\sqrt{2}$ with trivial
$\Z_2$-action. The singular set of $O_2$ consists of four copies of
$\Sp^1$ of length $1/\sqrt{2}$ with $\Z_2$-isotropy, so that the
twisted $\Z$-sectors consist of four copies of $\Sp^1$ of length
$1/\sqrt{2}$ equipped with trivial $\Z_2$-action.  Again, the
$\Z$-sectors have different numbers of connected components, and
hence $O_1$ and $O_2$ are not $\Z$-isospectral.  Similarly, for any
$\Gamma$, it is easy to see that $\widetilde{(O_1)}_\Gamma$ consists
of the nontwisted sector and $2(|\mathrm{HOM}(\Gamma,\Z_2)|-1)$
circles of length $\sqrt{2}$ with $\Z_2$-isotropy, while
$\widetilde{(O_2)}_\Gamma$ consists of the nontwisted sector and
$4(|\mathrm{HOM}(\Gamma,\Z_2)|-1)$ circles of length $1/\sqrt{2}$
with $\Z_2$-isotropy.  Therefore, $O_1$ and $O_2$ are not
$\Gamma$-isospectral for any $\Gamma$ that admits a nontrivial
homomorphism to $\Z_2$.

To see that that the asymptotic expansions of the $\Gamma$-heat kernels of
$O_1$ and $O_2$ coincide for every $\Gamma$, note that as $O_1$ and $O_2$
are isospectral, the usual heat kernels of $O_1$ and $O_2$ coincide.  In addition, the
asymptotic expansion of the heat kernel of a $1$-dimensional
manifold with connected components $M_1, \dots, M_n$ is given by
$(l_1 + \cdots +l_n)(4\pi t)^{- \frac{1}{2}}$ where $l_i$ is the length of $M_i$; see \cite[Section 9]{CahnWolf}
or \cite[Section 1.2]{poltervich}.  Hence, as the twisted
$\Gamma$-sectors of $O_1$ and $O_2$ consist of circles whose lengths
sum to $2\sqrt{2}(|\mathrm{HOM}(\Gamma,\Z_2)|-1)$, the contributions
of the twisted sectors to the asymptotic expansions of the
$\Gamma$-heat traces coincide.

% XXXXXXXXXXXXXXXXXXXXXXXXXXXXXXXXXXXXXXXXXXXXXXXXXXXXXXXXXXXXXX

\subsection{Lens space examples of Shams}
\label{subsec-shams} In \cite{shams}, Shams Ul Bari studies orbifold
lens spaces, orbifolds given by the quotient of the standard unit sphere by
a cyclic group of isometries.  Several pairs of isospectral,
nonisometric orbifolds are determined.  In each example, a pair of
orbifolds $O_1$ and $O_2$ is given of the form $G_1\ltimes \Sp^n$
and $G_2\ltimes \Sp^n$, respectively, where $G_1$ and $G_2$ are
cyclic groups of the same order acting as isometries on $\Sp^n$.  In
every example, the singular sets of $O_1$ and $O_2$ are identical,
given by spheres or products of spheres with the standard metric,
and the isotropy groups of these singular sets coincide.  It
therefore follows that the collection of twisted $\Gamma$-sectors of
$O_1$ is isometric to the collection of twisted $\Gamma$-sectors of
$O_2$ for any $\Gamma$, and hence each pair of isospectral lens
spaces is in fact $\Gamma$-isospectral for every $\Gamma$.

% XXXXXXXXXXXXXXXXXXXXXXXXXXXXXXXXXXXXXXXXXXXXXXXXXXXXXXXXXXXXXX
% XXXXXXXXXX    SECTION: SUNADA'S THEOREM
% XXXXXXXXXXXXXXXXXXXXXXXXXXXXXXXXXXXXXXXXXXXXXXXXXXXXXXXXXXXXXX

\section{The Sunada method and ${\Gamma}$-isospectrality}
\label{sec-Sunada}
Early examples of isospectral pairs of manifolds were produced using ad hoc arguments.
Sunada was the first to introduce a systematic method for producing isospectral manifolds,
\cite{sunada}.   His technique is based on identifying triples $(G, H_1,H_2)$ of finite
groups, with $H_1,H_2\leq G$, acting freely by isometries on a compact Riemannian manifold
$(M,g)$.  If $H_1$ and $H_2$ are \textit{almost conjugate} in $G$, meaning that each conjugacy
class in $G$ intersects $H_1$ and $H_2$ in the same number of elements, then $H_1\backslash M$
and $H_2\backslash M$ are isospectral manifolds.

In \cite{ikedagrassmann}, Ikeda gave a simple proof of Sunada's theorem that makes it evident that the
group $G$ can be any subgroup of the group of isometries of $(M,g)$, and that $H_1$ and $H_2$ need not
act freely.   (In his statement of the theorem, Ikeda assumed that $H_1$ and $H_2$ act freely, but did not
use the assumption in his proof.)  Thus we have the following.

\begin{theorem}[\cite{sunada}, \cite{ikedagrassmann}]\label{sunada}
Suppose that $(M,g)$ is a compact Riemannian manifold and that $G$ is a group that acts on $(M,g)$
on the left by isometries.  Suppose that $H_1$ and $H_2$ are finite, almost conjugate subgroups of $G$.
Then $O_1=H_1\backslash M$ and $O_2=H_2\backslash M$, with their respective submersion metrics, are
isospectral orbifolds.
\end{theorem}

We remark that if $H_1$ and $H_2$ are actually conjugate in $G$, then the resulting orbifolds
will be isometric.

In general, isospectral pairs or families of orbifolds arising from Sunada's method are not necessarily
$\Gamma$-isospectral for any particular choice of $\Gamma$.   We see this by noting that, as explained
in Section~\ref{sec-RSW}, none of the pairs of isospectral orbifolds in Examples 2.7, 2.8, 2.9, 3.7, or
3.9 in \cite{rsw}, all of which arise from an application of Sunada's theorem, are $\mathbb{Z}^\ell$-isospectral
for any positive $\ell$.  Similarly, although the orbifolds in any family of isospectral
orbifolds constructed by Shams, Stanhope, and Webb in \cite{ssw} are isospectral via Sunada's theorem,
they are pairwise not $\mathbb{Z}$-isospectral as demonstrated in Section~\ref{subsec-ssw}.

For given finitely generated discrete group $\Gamma$, in order to conclude that two Sunada-isospectral orbifolds are $\Gamma$-isospectral, we have the following.

\begin{theorem}\label{sunadasectors}
Let $(M,g)$ be a compact Riemannian manifold and $G$ a group acting on $(M,g)$ on the left by isometries.
Let  $H_1$ and $H_2$ be almost conjugate finite subgroups of $G$, and suppose that $\Gamma$ is a finitely
generated discrete group.  If there is a bijective correspondence between homomorphisms
$\varphi\co \Gamma\to H_1$ whose images have nonempty fixed point sets and homomorphisms
$\psi\co \Gamma\to H_2$ such that for each pair $\varphi,\psi$ there is an isometry
$i\co M^{\langle\varphi\rangle}\to M^{\langle\psi\rangle}$ and $C_{H_2}(\psi)$ is
almost conjugate to $iC_{H_1}(\varphi)i^{-1}$ in the isometry group of $M^{\langle\psi\rangle}$,
then $H_1\backslash M$ and $H_2\backslash M$ are $\Gamma$-isospectral.
\end{theorem}

\begin{proof}
Since $H_1$ and $H_2$ are almost conjugate, by Theorem~\ref{sunada}, $H_1\backslash M$ and $H_2\backslash M$ are isospectral.   Since $i\co  M^{\langle\varphi\rangle}\to M^{\langle\psi\rangle}$ is an isometry, also by Theorem~\ref{sunada}, $C_{H_1}(\varphi)\backslash M^{\langle\varphi\rangle}$ and $C_{H_2}(\psi)\backslash M^{\langle\psi\rangle}$ are isospectral orbifolds.
Thus, there is a bijective, isospectral correspondence between the sectors of $H_1\backslash M$ and $H_2\backslash M$, so by definition of the $\Gamma$-spectrum, $H_1\backslash M$ and $H_2\backslash M$ are $\Gamma$-isospectral.
\end{proof}

\begin{remark}
We note that since pairs of orbifolds arising from Theorem~\ref{sunada} are $p$-isospectral (i.e. are isospectral for the Laplace operator acting on $p$-forms) for all $p$, orbifolds arising from Theorem~\ref{sunadasectors} are $\Gamma$-isospectral on $p$-forms for all $p$ as well.
\end{remark}

We now construct a pair of orbifolds that have nontrivial $\Z$- and
$\Z^2$-sectors that are $\Gamma$-isospectral for all $\Gamma$ by
Theorem~\ref{sunadasectors}.

\begin{example}\label{ex-basedonrsw}
Let $K_1$ and $K_2$ denote the subgroups of $SO(6)$ defined in Section~\ref{subsec-rswHS},
Equations~\ref{eq-RSWGp1} and \ref{eq-RSWGp2}.
Recall that $K_1$ and $K_2$ are almost conjugate but not conjugate in $SO(6)$.
For $i=1,2$, let $K_i^{\Delta_2}$ denote the subgroup of $SO(12)$ isomorphic to $K_i$ given by identifying
$K_i$ with the diagonal in $K_i\times K_i < SO(12)$.

We define the orbifolds $O_1$ and $O_2$ as biquotients of $SO(15)$.
To begin, identify $SO(3)$ with the subgroup of $SO(15)$ consisting of matrices of the form
\begin{equation*}
\begin{bmatrix}
A& 0\\
0 & I_{12}
\end{bmatrix}
\end{equation*}
where $A\in SO(3)$.  Similarly, identify $SO(12)$ with the subgroup of $SO(15)$ of matrices
\begin{equation*}
\begin{bmatrix}
I_3 & 0\\
0 & B
\end{bmatrix}
\end{equation*}
where $B\in SO(12)$.  Using this identification, we may think of $K_i^{\Delta_2}<SO(12)$  as a subgroup of $SO(15)$.
Let $G_i<SO(15)$, $i=1,2$, be the subgroup isomorphic to $\mathbb{Z}_2^{5}$ generated by $a_{12}$, $a_{23}$,
and $K_i^{\Delta_2}$, and note that the $K_i^{\Delta_2}$ act on coordinates $4$ through $15$.
Furthermore, $G_1$ and $G_2$ are almost conjugate but not conjugate
in $SO(15)$ for the same reason that $K_1$ and $K_2$ are almost conjugate but not conjugate in $SO(6)$.

Let $M=SO(15)/SO(3)$. Then $M$ can be identified with the set of
$12$-frames in $\mathbb{R}^{15}$, where the $12$-frame associated to
the coset $bSO(3)$ of $b\in SO(15)$ is given by the last $12$
columns of $b$.  For $i = 1, 2$, let $O_i$ be the orbifold presented by
$G_i \ltimes M$ equipped with the submersion metric arising from a fixed biinvariant
metric on $SO(15)$.  By Theorem 2.5 in \cite{rsw}, $O_1$ and $O_2$ are isospectral orbifolds.

We now compute the sectors of the orbifolds $O_1$ and $O_2$.
Every element of $G_i$ is diagonal with eigenvalues $1$ or $-1$.  If
we identify an element $bSO(3)\in M$ with a $15\times 12$ matrix
having orthonormal columns, the left action of an element $h\in G_i$
on $bSO(3)$ negates the rows in $bSO(3)$ corresponding to the
positions of $-1$ on the diagonal in $h$.  Thus, for an element
$bSO(3)\in M$ to be fixed by $h$,  $bSO(3)$ must have a zero row
corresponding to the placement of each $-1$ in $h$.  For any $h\in
G_i$, we then see that $M^{\langle h\rangle}$ is the set of $12$-frames in
$\mathbb{R}^{15-m}$ where $m$ is the multiplicity of $-1$ as an
eigenvalue of $h$.  This implies that in order for an element of
$G_i$ to have a nonempty fixed point set, it must have eigenvalue
$-1$ with multiplicity no more than $3$.   We also note that by
construction, only even $m$ occur.

Since every element of $K_i^{\Delta_2}$ has eigenvalue $-1$ with
multiplicity of at least $4$, no element of $K_i^{\Delta_2}$ has
nonempty fixed point set in $M$.  Therefore, only the elements
$a_{12}$, $a_{13}$, and $a_{23}$ have nonempty fixed point sets.
Hence, the only subgroups of $G_i$ that have nonempty fixed point
set in $M$ are $\langle a_{12}\rangle$, $\langle a_{13}\rangle$, $\langle a_{23}\rangle$, and
$\langle a_{12}, a_{13}\rangle \cong (\Z_2)^2$.  Note that $M^{\langle
a_{12}\rangle}$, $M^{\langle a_{13}\rangle}$, and $M^{\langle a_{23}\rangle}$ correspond to the
collection of $12$-frames in $\R^{13}$, while $M^{\langle a_{12},
a_{13}\rangle}$ corresponds to the collection of $12$-frames in
$\R^{12}$. It follows that for any finitely generated discrete group
$\Gamma$, the bijection between homomorphisms $\varphi\co\Gamma\to G_1$
and $\psi\co\Gamma\to G_2$ with nonempty fixed point set required in
Theorem~\ref{sunadasectors} is trivial, as is the isometry $i\co
M^{\langle \varphi\rangle}\to M^{\langle\psi\rangle}$.  Then as
$C_{G_1}(\varphi) = G_1$ and $C_{G_2}(\psi) = G_2$ are almost
conjugate, it follows that $O_1$ and $O_2$ are $\Gamma$-isospectral
for all $\Gamma$.  Note that as $\langle a_{12},
a_{13}\rangle$ is not a homomorphic image of $\Z$, both $O_1$ and
$O_2$ have $\Z^2$-sectors that do not appear as $\Z$-sectors.

To show that $O_1$ and $O_2$ are not isometric, note that the lowest-dimensional
$\Z^2$-sectors of both $O_1$ and $O_2$ are the $66$-dimensional sectors
corresponding to homomorphisms with image $\langle a_{12}, a_{13} \rangle$.
They are given by $K_i^{\Delta_2} \backslash SO(12)$
for $i = 1, 2$, respectively.  Hence, it will be sufficient to show that
$K_1^{\Delta_2} \backslash SO(12)$ is not isometric to $K_2^{\Delta_2} \backslash SO(12)$.

Suppose for contradiction that $K_1^{\Delta_2} \backslash SO(12)$
and $K_2^{\Delta_2} \backslash SO(12)$ are isometric, and consider
the biquotients $O_1^\prime = K_1^{\Delta_2} \backslash SO(12)/K_1^{\Delta_2}$
and $O_2^\prime = K_2^{\Delta_2} \backslash SO(12)/K_1^{\Delta_2}$.  Then the manifold
$K_1^{\Delta_2} \backslash SO(12) = K_2^{\Delta_2} \backslash SO(12)$ is by hypothesis
a common Riemannian cover for both $O_1^\prime$ and $O_2^\prime$.
By \cite[Corollary 2.6]{rsw}, $O_1^\prime$ and $O_2^\prime$ are
isospectral and have different maximal isotropy orders.  In fact, computations similar
to those in Section \ref{subsubsec-rswSO3} demonstrate that both are noneffective
orbifolds with generic isotropy $\Z_2$.  The orbifold
$O_1^\prime$ contains points with isotropy $(\Z_2)^3$ that form strata of the
singular set of dimension $18$, while the lowest-dimensional strata of the
singular set of $O_2^\prime$ are of dimension $34$.

To see this, for each $g \in K_i$, let $\widetilde{g}$ denote the corresponding element of
$K_i^{\Delta_2}$.  We indicate these elements using coordinates in $SO(12)$
rather than $SO(15)$ for simplicity.  Then the element
$\widetilde{-I}$ acts trivially (on the right) on both $K_1^{\Delta_2} \backslash SO(12)$ and
$K_2^{\Delta_2} \backslash SO(12)$.  Similarly,
$\widetilde{a_{12}} \in K_1^{\Delta_2}$ fixes in both $K_1^{\Delta_2} \backslash SO(12)$
and $K_2^{\Delta_2} \backslash SO(12)$ components isometric to the set of right $K_i$-cosets of
matrices whose four $6\times 6$ blocks are of the form
\[
    \left[\begin{array}{c|c}
        \ast_{2\times 2}    &   0_{2\times 4}       \\
        \hline
        0_{4\times 2}       &   \ast_{4\times 4}
    \end{array}\right].
\]
This set of matrices in $SO(12)$ is diffeomorphic to $SO(4)\times SO(8)$, which is of dimension $34$,
and finitely covers the corresponding singular set in each orbifold.  In $O_2^\prime$, by an argument
similar to the one given in \cite[Example 2.8]{rsw}, these sets have
maximal isotropy $\langle \widetilde{I}, \widetilde{a_{12}}\rangle$ and hence are the
lowest-dimensional singular strata.  However, in $O_1^\prime$, the entire group
$K_1^{\Delta_2}$ fixes components corresponding to matrices
whose four $6\times 6$ blocks are of the form
\[
    \left[\begin{array}{ccc|ccc}
        \ast    &   0       &   0       &                       \\
        0       &   \ast    &   0       &      0_{3\times 3}    \\
        0       &   0       &   \ast    &                       \\ \hline
        & 0_{3\times 3}   &&    \ast_{3\times 3}
    \end{array}\right].
\]
Each such set of matrices is diffeomorphic to $SO(2)^3\times SO(6)$, which has dimension $18$, and finitely covers
the lowest-dimensional singular strata in $O_1^\prime$.

We conclude that $O_1^\prime$ and $O_2^\prime$ are isospectral orbifolds with a common Riemannian
cover such that the lowest-dimensional singular strata in each are of different dimensions.
This yields a contradiction when we demonstrate the following.
\end{example}

\begin{lemma}
\label{lem-GordRos}
Suppose $O_1^\prime$ and $O_2^\prime$ are isospectral orbifolds that have as a common Riemannian cover
the smooth manifold $M$.  Then the dimensions of the lowest-dimensional singular strata of
$O_1^\prime$ and $O_2^\prime$ coincide.
\end{lemma}
Note that the orbifolds $O_1^\prime$ and $O_2^\prime$ are required to be covered by a manifold and hence
are good orbifolds.  The proof is similar to that of \cite[Proposition 3.4(ii)]{GordonRossetti}.
\begin{proof}
Because $O_1^\prime$ and $O_2^\prime$ are isospectral, they have the same volume.  In the expression of the asymptotic
expansion of the heat kernel given in Equation \ref{eq-DGGW4.8byStrata}, the $a_k$
depend only on the volume of the orbifold and the curvature of $M$, so that the $a_k$ coincide
for $O_1^\prime$ and $O_2^\prime$.  It follows that the second terms in Equation \ref{eq-DGGW4.8byStrata} must coincide as well, i.e.
\[
    \sum\limits_{N_1^\prime\in\mathcal{S}(O_1^\prime)} (4\pi t)^{-\mathrm{dim}(N_1^\prime)/2}\sum\limits_{k=0}^\infty b_{k,N_1^\prime} t^k
    =
    \sum\limits_{N_2^\prime\in\mathcal{S}(O_2^\prime)}
    (4\pi t)^{-\mathrm{dim}(N_2^\prime)/2}\sum\limits_{k=0}^\infty b_{k,N_2^\prime}t^k,
\]
where the sums are again over the singular strata of the orbifolds and the $b_{k,N_i^\prime}$ are the coefficients
for the strata of $O_i^\prime$.
However, as $b_{0,N_i^\prime} \neq 0$ for each
$N_i^\prime \in \mathcal{S}(O_i^\prime)$, $i=1,2$,
and since the lowest-degree terms must coincide, the claim follows.
\end{proof}

\begin{remark}
By \cite[Proposition 3.2]{farseageneuler}, the $\Gamma$-sectors of a
product orbifold $O \times O^\prime$ are given by the products of
the sectors of $O$ and $O^\prime$.  Clearly, if $O_1$ and $O_2$
satisfy the hypotheses of Theorem~\ref{sunada}, then so do $O\times
O_1$ and $O\times O_2$ for any fixed (quotient) orbifold $O$.
Therefore, by taking the product of the orbifolds in
Example~\ref{ex-basedonrsw} with an orbifold $O$ that has $\Z^\ell$-sectors
that do not appear as $\Z^{\ell-1}$-sectors, the resulting orbifolds
are $\Gamma$-isospectral for all $\Gamma$ and have
$\Z^{\ell}$-sectors that do not appear as $\Z^{\ell-1}$-sectors.
\end{remark}

In the next example, we consider the isospectral deformation of orbifolds found in \cite{procstan}.  We recall that these orbifold were found using the following generalization of the Sunada theorem in \cite{dg}, recast in the orbifold setting in \cite{procstan}.  We will show that any pair of orbifolds in the deformation are $\Gamma$-isospectral for any $\Gamma$.   While we will not be able to prove this using a direct application of Theorem~\ref{sunadasectors} because the groups involved are not finite, the philosophy will be the same.  We will show that there is a bijection between $\Gamma$-sectors such that corresponding sectors are isospectral; we will in fact show that they are isometric.

For a Lie group $G$ with subgroup $H$, we say that an automorphism $\Phi\co G\to G$
is an \textit{almost-inner automorphism relative to $H$} if for all $h\in H$ there is
an element $a\in G$ such that $\Phi(h)=aha^{-1}$.

\begin{theorem}[\cite{dg}, \cite{procstan}]\label{fancysunada}
Suppose that $G$ is a Lie group with simply connected, nilpotent identity component $G_0$.  Let $H$ be a discrete subgroup of $G$ such that $G=H G_0$ and $(G_0\cap H)\backslash G_0$ is compact.  Suppose that $G$ acts effectively and properly discontinuously on the left by isometries on $(M,g)$ with $H\backslash M$ compact.  Let $\Phi\co G\to G$ be an almost-inner automorphism relative to $H$.  Then, letting $g$ denote the submersion metric, the quotient orbifolds $(H\backslash M,g)$ and $(\Phi(H)\backslash M, g)$ are isospectral.
\end{theorem}

\begin{example}\label{ex-basedonprocstan}
Let $G$ be the Lie group
\begin{equation*}
\{(x_1,x_2,y_1,y_2,z_1,z_2)\, \vert \, x_i,y_i,z_i\in \R\}
\end{equation*}
with group multiplication given by
\begin{multline*}
(x_1,\dots,z_2)(x_1\p,\dots, z_2\p) \\
=(x_1+x_1\p, \dots, y_2+y_2\p, z_1+z_1\p+x_1y_1\p+x_2y_2\p, z_2+z_2\p+x_1y_2\p).
\end{multline*}

We denote elements of $\mathrm{Aut}(G)\ltimes G$ by ordered pairs $(\psi, \bar{x})$.  The group multiplication in $\mathrm{Aut}(G)\ltimes G$ is given by $(\psi, \bar{x})(\psi^\prime, \bar{x}^\prime) = (\psi\psi^\prime, \bar{x}\psi(\bar{x}^\prime))$ and $\mathrm{Aut}(G)\ltimes G$ acts on $G$ by $(\psi, \bar{x})\cdot\bar{g}= \bar{x}\psi(\bar{g})$.

Suppose that $\La$ is the integer lattice in $G$.  Let $\alpha\in \mathrm{Aut}(G)\ltimes G$ be
given by the ordered pair $(\varphi, (0,0,0,0,0,\frac{1}{2}))$, where $\varphi$ is the element of
$\mathrm{Aut}(G)$ given by
$\varphi(x_1,x_2,y_1,y_2,z_1,z_2) = (x_1,x_2,-y_1,-y_2,-z_1,-z_2)$.   Then
\begin{equation*}
\alpha(x_1,x_2,y_1,y_2,z_1,z_2) = (x_1,x_2,-y_1,-y_2,-z_1,-z_2+\tfrac{1}{2}).
\end{equation*}
Define $H= \La\cup \alpha\La$.

For $t\in [0,1)$ define an automorphism $\Phi_t\co G\to G$ by
\begin{equation*}
\Phi_t(x_1,x_2,y_1,y_2,z_1,z_2) =(x_1,x_2,y_1,y_2,z_1,z_2+ty_2).
\end{equation*}
From \cite{dg}, $\Phi_t$ is an almost-inner automorphism of $G$ relative to $\La$.  In particular, letting $a=(0,0,0,0,0,0)$ if
$y_2=0$ and $(t,-\tfrac{ty_1}{y_2},0,0,0,0)$ otherwise, we see that $\Phi_t(x)=ax a^{-1}$ for any $x\in G$.
Recalling that $HG=G\cup\alpha G$, extend $\Phi_t$ to an automorphism $\tilde{\Phi}_t\co HG\to HG$ by setting
$\tilde{\Phi}_t(x) = \tilde{\Phi}_t(\mathrm{Id},x) = (\mathrm{Id}, \Phi_t(x))$ and
$\tilde{\Phi}_t(\alpha x) = \tilde{\Phi}_t(\varphi,\alpha\cdot x) =(\varphi, \Phi_t(\alpha\cdot x))= (\varphi, \alpha\cdot\Phi_t(x))$.
Then $\tilde{\Phi}_t$ is an almost-inner automorphism relative to $H$; indeed, for $h=x$ or $\alpha x\in H$,
$\tilde{\Phi}_t(h) = (\mathrm{Id},a)h(\mathrm{Id},a^{-1})$ where $a$ is based on $x$ as above.  Therefore,
letting $g$ be an $HG$-invariant metric on $G$, by Theorem~\ref{fancysunada}, we have a continuous isospectral
family of orbifolds, $(\tilde{\Phi}_t(H)\backslash G, g)$, $t\in[0,1)$.  By \cite{procstan}, the deformation is nontrivial.

We now show that this deformation is $\Gamma$-isospectral for all $\Gamma$. We begin by computing the sectors of
$(\tilde{\Phi}_t(H)\backslash G, g)$.    Since $H=\La\cup\alpha\La$, $\tilde{\Phi}_t(H)= \tilde{\Phi}_t(\La)\cup\tilde{\Phi}_t(\alpha\La)$.
Elements of $\tilde{\Phi}_t(\La)$ are elements of $G$ and thus have empty fixed point sets.  Consider $\tilde{\Phi}_t(\alpha \lambda)$ where
$\lambda\in \La$.  If $\lambda= (a,b,c,d,e,f)$, then as an ordered pair in $\mathrm{Aut}(G)\ltimes G$,
\begin{equation*}
\tilde{\Phi}_t(\alpha\lambda) = (\varphi, (a,b,-c,-d,-e,-f+\tfrac{1}{2}-td)).
\end{equation*}
For $x=(x_1,x_2,y_1,y_2,z_1,z_2)\in G$, direct computation shows that $\tilde{\Phi}_t(\alpha\lambda)\cdot x = x$ if and only if
$a=b=0$, $c=-2y_1$, $d=-2y_2$, $e=-2z_1$, and $f=-2z_2+\tfrac{1}{2}-td$.  In this case, for
$\lambda =(0,0,c,d,e,f)$ with $c,d,e,f\in\mathbb{Z}$, the fixed
point set of $\tilde{\Phi}_t(\alpha\lambda)$ is
\begin{equation*}
G^{\langle\tilde{\Phi}_t(\alpha\lambda)\rangle}=\{(x_1,x_2,-\tfrac{c}{2}, -\tfrac{d}{2}, -\tfrac{e}{2}, -\tfrac{f}{2}+\tfrac{1}{4}-\tfrac{td}{2})\,\vert\, x_1,x_2\in\R \}.
\end{equation*}

We note that if $\lambda=(0,0,c,d,e,f)$, then the order of $\tilde{\Phi}_t(\alpha\lambda)$ in
$\mathrm{Aut}(G)\ltimes G$ is $2$.  Moreover, for $\lambda\neq\lambda\p$, the fixed point sets
of $\tilde{\Phi}_t(\alpha\lambda)$ and $\tilde{\Phi}_t(\alpha\lambda\p)$ do not intersect.
Therefore, the only nontrivial isotropy groups in $\tilde{\Phi}_t(H)\backslash G$ are
$\{\mathrm{Id}, \tilde{\Phi}_t(\alpha\lambda)\}\cong\mathbb{Z}_2$.  This implies that
if $\Gamma$ is a group that admits $\mathbb{Z}_2$ as a homomorphic image, the $\Gamma$-sectors of
$\tilde{\Phi}_t(H)\backslash G$ will all be of the form
$C_{\tilde{\Phi}_t(H)}(\tilde{\Phi}_t(\alpha\lambda))\backslash G^{\langle\tilde{\Phi}_t(\alpha\lambda)\rangle}$.
If $\Gamma$ does not admit $\Z_2$ as a homomorphic image, $\tilde{\Phi}_t(H)\backslash G$ has no nontrivial
$\Gamma$-sectors.

For $t\in [0,1)$ and for fixed $\lambda=(0,0,c,d,e,f)$, the action of the element
$i=(\mathrm{Id},(0,0,0,0,0,-\tfrac{td}{2}))\in\mathrm{Aut}(G)\ltimes G$ maps
$G^{\langle\alpha\lambda\rangle}$ to $G^{\langle\tilde{\Phi}_{t}(\alpha\lambda)\rangle}$.
Since $G^{\langle\tilde{\Phi}_{t}(\alpha\lambda)\rangle}$ is a totally geodesic submanifold
of $G$, the metric on $G^{\langle\tilde{\Phi}_{t}(\alpha\lambda)\rangle}$ is given by the
restriction of the metric from $G$.  Since $i\in G<HG$, and the metric on $G$ is $HG$-invariant,
$i$ is an isometry from $G^{\langle\alpha\lambda\rangle}$ to $G^{\langle\tilde{\Phi}_{t}(\alpha\lambda)\rangle}$.

For $\lambda=(0,0,c,d,e,f)$ and $t\in[0,1)$, the centralizer of $\tilde{\Phi}_t(\alpha\lambda)$ in $\tilde{\Phi}_t(H)$
is $\{(\mathrm{Id}, (p, q, 0, 0, \tfrac{cp}{2}+\tfrac{dq}{2}, \tfrac{dp}{2}))\}\cup \{(\varphi, (p, q ,-c,-d,-e-\tfrac{cp}{2}-\tfrac{dq}{2}, \tfrac{1}{2}-f-\tfrac{dp}{2}-dt))\}$
where $p,q\in\mathbb{Z}$ are such that $cp + dq$ and $dp$ are both even.  (Note that $C_{H}(\alpha\lambda)$ corresponds to $t=0$.)  For $i=(\mathrm{Id}, (0,0,0,0,0,\tfrac{-td}{2})$ as above, direct computation shows that
$iC_H(\alpha\lambda)i^{-1} = C_{\tilde{\Phi}_t(H)}(\tilde{\Phi}_t(\alpha\lambda))$.  Since these groups are in fact equal, for any $t\in [0,1)$, the sectors
$C_H(\alpha\lambda)\backslash G^{\langle\alpha\lambda\rangle}$ and $C_{\tilde{\Phi}_t(H)}(\tilde{\Phi}_t(\alpha\lambda)\backslash G^{\langle\tilde{\Phi}_{t}(\alpha\lambda)\rangle}$
are isometric, hence isospectral.
Therefore, the collection $\tilde{\Phi}_t(H)\backslash G$, $t\in[0,1)$, is a continuous deformation of orbifolds that are $\Gamma$-isospectral for all $\Gamma$.

\end{example}

For our final example, we construct a pair of 5-dimensional flat orbifolds that are $\Gamma$-isospectral for all $\Gamma$.  Here again, since the groups involved are not finite, we will not be able to apply Theorem~\ref{sunadasectors} directly but we will use the same idea that underlies the proof of that theorem: after proving that the orbifolds themselves are isospectral using the eigenvalue counting method of Miatello and Rossetti, we will show that there is a bijection between $\Gamma$-sectors such that corresponding sectors are isospectral.

\begin{example}\label{ex-torus}
Let $L_1$ and $L_2$ be a pair of 4-dimensional isospectral nonisometric lattices found in Section 2 of \cite{cs}.  Orthogonally extend $L_1$ and $L_2$ by vectors of the equal length to 5-dimensional isospectral, nonisometric lattices $\La_1$ and $\La_2$ (see \cite[p.154]{bgm}).

Let $g$ be the isometry of $\R^5$ given by reflection across the copy of $\R^4$ that contains $L_1$ and $L_2$.  For $i=1,2$, let $G_i$ be the subgroup of $O(5)\ltimes \R^5$ generated by $g$ and $\La_i$.  Letting $T_{\La_i}=\La_i\backslash \R^5$ and $O_i =G_i\backslash \R^5$, $T_{\La_i}$ covers $O_i$ and $O_i\cong \overline{G_i}\backslash T_{\La_1}$ where $\overline{G_i}:=G_i/\La_i$.  (See \cite[p.357]{rsw}.)  Letting $F$ denote the projection of $G_i$ onto $O(5)$, by the first isomorphism theorem, $F=\{\mathrm{Id}, g\}$ is isomorphic to $\overline{G}_i$.  Thus we can identify $O_i$ with $\Z_2\backslash T_{\La_i}$.

We confirm that $O_1$ and $O_2$ are isospectral using Miatello and Rosetti's eigenvalue counting formula, Theorem 3.1 in \cite{rsw}.  For any eigenvalue $\mu$, the multiplicity of $\mu$ in the spectrum of $O_i$ is given by
\begin{equation*}
d_{\mu}(G_i) = (\#F)^{-1}\sum_{B\in F} e_{\mu,B}(G_i), \textrm{ where }e_{\mu,B}(G_i) :=\sum_{v\in\La_i^*, \Vert v\Vert^2=\mu, Bv=v} e^{2\pi i\langle v,b\rangle}
\end{equation*}
and $b$ is chosen so that $BL_b\in G_i$.

It is straightforward to compute $d_{\mu}(G_i)$.  When $B=\mathrm{Id}$, let $b$ be any element of $\La_i$.  Since $\langle v,b\rangle\in \mathbb{Z}$ for all $v\in \La_i^*$,
\begin{equation*}
e_{\mu,\mathrm{Id}}(G_i) = \# \{v\in\La_i^*\, \vert\, \Vert v\Vert^2=\mu \}
\end{equation*}
i.e. $e_{\mu,\mathrm{Id}}(G_i)$ is the multiplicity of $\mu$ as an eigenvalue of $T_{\La_i}$.    Since $T_{\La_1}$ and $T_{\La_2}$ are isospectral, $e_{\mu,\mathrm{Id}}(G_1) = e_{\mu,\mathrm{Id}}(G_2)$.

When $B=g$, note that the elements of $\La_i^*$ that are fixed by $B$ are the elements of $L_i$.  Thus $e_{\mu,g}(G_i)$ is equal to the multiplicity of $\mu$ as an eigenvalue of $L_i\backslash\R^4$.  Since $L_1\backslash \R^4$ and $L_2\backslash\R^4$ are isospectral, $e_{\mu,g}(G_1)=e_{\mu,g}(G_2)$ for all $\mu$.

Therefore for any eigenvalue $\mu$, $d_{\mu}(G_1)=d_{\mu}(G_2)$ so $G_1\backslash\R^5$ and $G_2\backslash \R^5$ are isospectral orbifolds.

We now confirm that these orbifolds are $\Gamma$-isospectral for any $\Gamma$.
Since for $i=1,2$ the lattice $\La_i$ acts on $\R^5$ by translation, the only
finite subgroups of $G_i$ are of the form $\{\mathrm{Id}, (g,(0,0,0,0,e))\}\cong \Z_2$
where $(0,0,0,0,e)\in \La_i$. Thus for any finitely generated discrete group $\Gamma$, either
$\Gamma$ admits $\mathbb{Z}_2$ as a homomorphic image or it does not.  If it does, $O_i$ will
have a nontrivial $\Gamma$-sector for each nontrivial $\varphi\in\mathrm{HOM}(\Gamma, G_i)$ having image $\{\mathrm{Id}, (g,(\bar{0},e))\}$.
Since $\La_1$ and $\La_2$ are extensions of $L_1$ and $L_2$ by the same vector in $\R^5$ which is
orthogonal to both $L_1$ and $L_2$, for any homomorphism $\varphi\co \Gamma\to G_1$, there is an
obvious corresponding homomorphism $\psi\co \Gamma\to G_2$ that has the same image as $\varphi$.

If $\varphi(\Gamma)= \{\mathrm{Id}, (g,(\bar{0},e))\}$, then the fixed point set of  the image $\varphi(\Gamma)$
of $\varphi$ is $(\R^5)^{\langle\varphi\rangle} =\{(x,y,z,w,\tfrac{e}{2})\, \vert\, x,y,z,w\in\R\}$.
The centralizer of $\varphi(\Gamma)$ in $G_i$ is
$C_{G_i}(\varphi)=\{(\mathrm{Id},(\bar{a}, 0))\, \vert\, \bar{a}\in L_i\}\cup\{(g,(\bar{a},e))\,\vert\, \bar{a}\in L_i\}$.
A typical $\Gamma$-sector in $O_i$ is of the form $C_{G_i}(\varphi)\backslash (\R^5)^{\langle\varphi\rangle}$.

We note that for $i=1$ or $2$,  all of the $\Gamma$-sectors of $O_i$ are isometric to each other.  Indeed if $\varphi(\Gamma)=\{\mathrm{Id},(g,(0,0,0,0,0))\}$, then for any other homomorphism $\varphi^{\p}\co \Gamma\to G_i$ with image $\{\mathrm{Id}, (g,(\bar{0}, e))\}$, translation by $p=(\bar{0},\tfrac{e}{2})$ is an isometry from $(\R^5)^{\langle\varphi\rangle}$ to $(\R^5)^{\langle\varphi^{\p}\rangle}$ and $L_pC_{G_i}(\varphi)L_p^{-1} = C_{G_i}(\varphi^{\p})$.  Therefore the sectors $C_{G_i}(\varphi)\backslash (\R^5)^{\langle\varphi\rangle}$ and  $C_{G_i}(\varphi^{\p})\backslash (\R^5)^{\langle\varphi^{\p}\rangle}$ are isometric.

Finally, for any choice of $\Gamma$ that admits $\Z_2$ as a homomorphic image, the corresponding $\Gamma$-sectors of $O_1$ are $O_2$ are isospectral as follows.   Suppose that $\varphi\co \Gamma\to G_1$ has image $\{\mathrm{Id}, (g,(0,0,0,0,0))\}$.  The corresponding homomorphism $\psi\co \Gamma\to G_2$ has the same image.   The images of both $\varphi$ and $\psi$ also have the same fixed point sets, namely the copy of $\R^4$ fixed by the action of $g$ on $\R^5$.  Recalling that $F=\{\mathrm{Id},g\}$, the centralizer of $\varphi(\Gamma)$ in $G_1$ is given by $F\ltimes L_1$ and the centralizer of $\psi(\Gamma)$ in $G_2$ is given by $F\ltimes L_2$.   Since $L_1$ and $L_2$ are isospectral, by an argument similar to the one given above using Miatello and Rossetti's eigenvalue counting formula, $(F\ltimes L_1)\backslash \R^4$ is isospectral to $(F\ltimes L_2)\backslash \R^4$.  Since all other sectors in $O_i$ for $i=1$ or $2$ are isometric to these sectors respectively, we conclude that $O_1$ and $O_2$ are $\Gamma$-isospectral.

\end{example}

\bibliographystyle{amsplain}

\end{document}